\documentclass{amsart}
\usepackage{tikz}
\usepackage{amsmath,amsfonts,amsthm,amssymb,amscd, verbatim, graphicx,textcomp, hyperref}

\usepackage{pinlabel}
\usepackage{lscape}
\usepackage{tabularx}
\usepackage{latexsym}
\usepackage{multirow}
\newtheorem{Thm}{Theorem}[section]
\newtheorem{Cor}[Thm]{Corollary}
\newtheorem{Conj}{Conjecture}
\newtheorem{Prop}[Thm]{Proposition}

\newtheorem{Lem}[Thm]{Lemma}

\newtheorem*{thma}{Theorem A}
\newtheorem*{thmb}{Theorem B}
\newtheorem*{thmc}{Theorem C}
\newtheorem*{thmd}{Theorem D}

\newtheorem*{thme}{Theorem E}

\theoremstyle{definition}
\newtheorem{Def}[Thm]{Definition}

\theoremstyle{remark}

\numberwithin{equation}{section}

\newcommand{\Aut}{\operatorname{Aut}}

\newcommand{\Mor}{\operatorname{Mor}}
\newcommand{\stab}{\operatorname{stab}}

\renewcommand{\dim}{\operatorname{dim}}

\newcommand{\De}{\mathcal{D}}

\newcommand{\Sym}{\operatorname{Sym}}
\newcommand{\supp}{\operatorname{supp}}
\newcommand{\Alt}{\operatorname{Alt}}
\newcommand{\Sp}{\operatorname{Sp}}
\newcommand{\PSp}{\operatorname{PSp}}

\renewcommand{\Gamma}{\varGamma}
\renewcommand{\epsilon}{\varepsilon}

\renewcommand{\leq}{\leqslant}
\renewcommand{\geq}{\geqslant}

\newcommand{\I}{\mathcal{I} }
\renewcommand{\S}{\mathcal{S} }
\newcommand{\B}{\mathcal{B} }

\newcommand{\ep}{\epsilon}
\newcommand{\X}{\mathcal{X} }
\newcommand{\Y}{\mathcal{Y} }

\renewcommand{\B}{\mathcal{B}}

\renewcommand{\L}{\mathcal{L}}

\newcommand{\T}{\mathcal{T}}
\newcommand{\U}{\mathcal{U}}

\renewcommand{\L}{\mathcal{L}}
\newcommand{\C}{\mathcal{C}}

\renewcommand{\O}{\mathcal{O}}

\begin{document}

%%
%% The title of the paper goes here.  Edit to your title.
%%

\title{Conway groupoids and Completely Transitive Codes}
 
%%
%% Now edit the following to give your name and address:
%% 

\author{Nick Gill}\thanks{Part of the work for this paper was completed while Nick Gill was a visiting professor at the Universidad de Costa Rica. He would like to thank the mathematics department there for their warm hospitality. In addition, all three authors would like to thank Professor Noam Elkies for answering our questions about $M_{13}$.}
\address{Department of Mathematics,
University of South Wales,
Treforest, CF37 1DL
}
\email{nickgill@cantab.net}
\author{Neil I. Gillespie}
\address{Heilbronn Institute for Mathematical Research, Department of Mathematics, University of Bristol, U.K.}
\email{neil.gillespie@bristol.ac.uk}
%\author{Anthony Nixon}
%\address{Department of Mathematics and Statistics, York University, Canada}
%\email{tnixon@mathstat.yorku.ca}
\author{Jason Semeraro}
\address{Heilbronn Institute for Mathematical Research, Department of Mathematics, University of Bristol, U.K.}
\email{js13525@bristol.ac.uk}

%%
%% If there are three of more authors they are added in the obvious
%% way. 
%%

%%%
%%% The following is for the abstract.  The abstract is optional and
%%% if not used just delete, or comment out, the following.
%%%

\begin{abstract}
To each supersimple $2-(n,4,\lambda)$ design $\De$ one associates a `Conway groupoid,' which may be thought of as a natural generalisation of Conway's Mathieu groupoid associated to $M_{13}$ which is constructed from $\mathbb{P}_3$.

We show that $\Sp_{2m}(2)$ and $2^{2m}.\Sp_{2m}(2)$ naturally occur as Conway groupoids associated to certain designs. It is shown that the incidence matrix associated to one of these designs generates a new family of completely transitive $\mathbb{F}_2$-linear codes with minimum distance 4 and covering radius 3, whereas the incidence matrix of the other design gives an alternative construction to a 
previously known family of completely transitive codes.

We also give a new characterization of $M_{13}$ and prove that, for a fixed $\lambda > 0,$ there are finitely many Conway groupoids for which the set of morphisms does not contain all elements of the full alternating or symmetric group.
\end{abstract}



\maketitle

\section{Introduction}\label{s:intro}
In recent work with A.~Nixon \cite{Puzz}, we introduced the notion of a \textit{Conway groupoid}. To construct such an object we start with a supersimple $2-(n,4,\lambda)$ design $\De$ i.e. a design for which any two lines intersect in at most two points. The Conway groupoid $\C(\De)$ of $\De$ is a small category whose set of objects is the set of points in $\De$, and whose morphisms can be `read off' from the lines in $\De$; in particular, this process associates an element of the group $\Sym(n)$ to each morphism in $\C(\De)$ (see Section \ref{s:back} for full details).

The concept of a Conway groupoid is a direct generalization of the groupoid associated to Conway's famous construction of $M_{13}$ using a `game' played on $\mathbb{P}_3$, the finite projective plane of order $3$ \cite{Co1}. Thus, by viewing $\mathbb{P}_3$ as a supersimple $2-(13,4,1)$ design, the set $M_{13}$ inside $\Sym(13)$ determines a Conway groupoid. A number of other examples were constructed in \cite{Puzz}. In this paper we are interested in constructing more examples of Conway groupoids and in working towards a full classification. 

{\bf Constructing examples}: We say that a Conway groupoid $\C$ associated to a supersimple $2-(n,4,\lambda)$ design is \textit{full} if every element of $\Alt(n)$ occurs as a morphism. One of the main results in \cite{Puzz} suggests that those designs whose Conway groupoids are not full are rare \cite[Theorem C]{Puzz}. In this paper, we demonstrate the existence of two infinite families of designs with this property. These families arise from:
\begin{itemize}
\item[(a)] the two 2-transitive actions of $\Sp_{2m}(2)$ on sets of quadratic forms over a $2m$-dimensional $\mathbb{F}_2$-vector space, for $m \geq 3$;
\item[(b)] the natural action of the affine group $2^{2m}.\Sp_{2m}(2)$ on $(\mathbb{F}_2)^{2m}$.
\end{itemize}
Both actions give rise to codes associated to the incidence matrices of the corresponding designs. In case (b), these codes were already known (see \cite{BRZ}). However, in case (a) the codes which arise are new; they are completely transitive and have covering radius $3$. 

{\bf Classifying Conway groupoids}: We prove two main results - Theorems D and E below - that give classifications of Conway groupoids subject to certain extra suppositions. Both results have interesting implications: Theorem D gives a new characterization of the Conway groupoid determined by $M_{13}$; Theorem E yields a proof of \cite[Conjecture 8.1]{Puzz}, which asserts that for each $\lambda > 0$ there exist only finitely many supersimple $2-(n,4,\lambda)$ designs whose Conway groupoids are not full.

\subsection{The main theorems}

In this section we briefly outline the main results of the paper. The definitions of all terms used in the statement of these results can be found in Sections~\ref{s:back} and \ref{s:action}.

In order to construct new infinite families of Conway groupoids we study the action of the group $\Sp_{2m}(2)$ on the set $\Omega$ of quadratic forms $(\mathbb{F}_2)^{2m}\to\mathbb{F}_2$. We make use of a bijection between $\Omega$ and the vector space $V\cong\mathbb{F}_2^{2m}$ on which $\Sp_{2m}(2)$ naturally acts,
allowing us to denote quadratic forms by $\theta_a$ for some $a\in V$. (This bijection is fully explained in Section~\ref{s:action}.) 

For $\ep \in \{0,1\}$ write $V^\ep:=\{v \in V \mid \theta_0(v)=\ep\}$ (here $\theta_0$ is the quadratic form associated with the zero vector). Then the induced action of $\Sp_{2m}(2)$ on $\Omega$ splits into two orbits $\Omega^0$ and $\Omega^1$ where $\Omega^\ep:=\{\theta_a \mid a \in V^\ep\}$. Our first result asserts the existence of some supersimple designs with these orbits.%gives an explicit description of certain orbits on sets of size 4 for each of these actions, together with the natural action of $\Sp_{2m}(2)$ of $V$.

\begin{thma}
Let $m \geq 3$ and $\epsilon \in \mathbb{F}_2$. Then  $$\B^\ep:=\{\{\theta_{v_1},\theta_{v_2},\theta_{v_3},\theta_{v_1+v_2+v_3} \} \mid v_i \in V^\ep, \sum_{i=1}^3 v_i \in V^\ep\}$$ forms the line set for a supersimple $2-(f_\epsilon(m),4,f_\epsilon(m-1)-1)$ design $(\Omega^\ep,\B^\ep)$ where  
\begin{equation}\label{e:fep}
f_\epsilon(m):=|\Omega^\epsilon|=2^{m-1} \cdot (2^m + (-1)^\epsilon).
\end{equation}
Furthermore, letting $\theta_0$ be the quadratic form associated to the zero vector, $$\B^a:=\{\{v_1,v_2,v_3,v_1+v_2+v_3\} \mid v_i \in V, \sum_{i=1}^3 \theta_0(v_i)=\theta_0\left(\sum_{i=1}^3 v_i\right)\}$$ forms the line set for a supersimple $2-(2^{2m},4,2^{2(m-1)}-1)$ design $(V,\B^a)$. 
\end{thma}

%\begin{thma}
%Let $m \geq 3$ and Let $\Omega^0$, $\Omega^1$ be the two orbits of $\Sp_{2m}(2)$ under its natural action on quadratic forms. Then for each $\epsilon \in \mathbb{F}_2$, the action of $\Sp_{2m}(2)$  on $3$-subsets of elements of $\Omega^\epsilon$ splits into two orbits, $\O_0^\epsilon$ and $\O_1^\epsilon$ and $\Sp_{2m}(2)$ acts imprimitively on each of these. Furthermore, writing $$\B^\ep:=\{\{\theta_a,\theta_b,\theta_c,\theta_{a+b+c}\} \mid \{\theta_a,\theta_b,\theta_c\} \in \O_\ep^\ep\},$$ we have that $\De^\ep:=(\Omega^\ep,\B^\ep)$ is a supersimple 
%\end{thma}

Let us write $2^{2m}$ for an elementary abelian group of order $2^{2m}$. Then, as we shall see (Lemmas~\ref{l:spdes} and \ref{l:aff}), the sets $\B^\ep$ (resp. $B^a$) are, in fact, orbits of the group $\Sp_{2m}(2)$ (resp. $2^{2m}.\Sp_{2m}(2)$) acting on the set of 4-subsets of $\Omega$ (resp. $V$). It turns out that this is not the first time that the action of $\Sp_{2m}(2)$ on the set of associated quadratic forms has been used to construct designs with special properties \cite{wertheimer}.

Write $\De^\ep:=(\Omega^\ep,\B^\ep)$ and $\De^a:=(V,\B^a)$ for the designs constructed in Theorem A. Our next result, Theorem B, proves the existence of an infinite class of Conway groupoids; these are the Conway groupoids associated to $\De^a$ and $\De^\ep$.

To understand the statement of the theorem we recall that, given a point $\infty$ in a design $\De$, we write $\L_\infty(\De)$ for the set of all move sequences in $\De$ which 
start at $\infty$, while we write $\pi_\infty(\De)$ for the set of all move sequences which start and end at $\infty$. Writing $n$ for the number of points in $\De$, we observe that 
$\L_\infty(\De)$ is a subset of $\Sym(n)$, while $\pi_\infty(\De)$ is a subgroup of $\Sym(n-1)$ which we call the {\it hole stabilizer}. In Section \ref{sub:conway}, we describe 
how the Conway groupoid $\mathcal{C}(\De)$ is completely determined by $\L_\infty(\De)$, which explains the focus of the following theorem (and indeed the focus of
Theorems D and E).

\begin{thmb}
Let $m \geq 3$ and let $\De^a$ and $\De^\ep$ be as above. The following hold:

\begin{itemize}
\item[(a)] Let $\infty$ be a point in $\De^\ep$. Then $\L_\infty(\De^\ep)$ coincides with a subgroup of $\Sym(\Omega^\ep)$ isomorphic to $\Sp_{2m}(2)$ and $\pi_\infty(\De^\ep)$ coincides with the stabilizer of $\infty$ inside $\L_\infty(\De^\ep)$;
\item[(b)]  Let $\infty$ be a point in $\De^a$. Then $\L_\infty(\De^a)$ coincides with a subgroup of $\Sym(V)$ isomorphic to  $2^{2m}.\Sp_{2m}(2)$ and $\pi_\infty(\De^a)$ coincides with the stabilizer of $\infty$ inside $\L_\infty(\De^a)$.
\end{itemize}
\end{thmb}

Recall that to any design $\De$ and prime $p > 0$ one may associate the code $C_{\mathbb{F}_p}(\De)$,
the $\mathbb{F}_p$-rowspan of the incidence matrix of $\De$. In \cite{Puzz}, using GAP \cite{GAP} 
we constructed examples of non-full Conway groupoids whose hole stabilizer is a primitive subgroup of $\Sym(n-1)$. In each case we also constructed $C_{\mathbb{F}_p}(\De)$ for $p=2$ or $3$, and discovered that the code was \emph{completely transitive} and, therefore, also \emph{completely regular} (see Definitions \ref{d:ctrans} and \ref{d:creg} below).  

The following result, Theorem C, asserts that the same is true of the $\mathbb{F}_2$-linear codes $C_{\mathbb{F}_2}(\De^\ep)$ and $C_{\mathbb{F}_2}(\De^a)$  constructed using the designs considered in Theorem A. Theorem C also describes the covering radius and intersection array of these codes (see Definition \ref{d:creg}). Recall that the function $f_\ep$ is defined at \eqref{e:fep}.

\begin{thmc}
Let $m \geq 3$ and let $\De^\ep$ and $\De^a$ be as above. The following hold:

\begin{itemize}
\item[(a)] $C_{\mathbb{F}_2}(\De^\epsilon)$ is a completely transitive $[f_\epsilon(m),f_\epsilon(m)-(2m+1),4]$ code with covering radius $3$ and intersection array $$(f_\epsilon(m),f_\epsilon(m)-1,f_\epsilon(m)-2f_\ep(m-1);1,2f_\ep(m-1),f_\ep(m)).$$
\item[(b)] $C_{\mathbb{F}_2}(\De^a)$ is a completely transitive $[2^{2m},2^{2m}-(2m+2),4]$ code with covering radius $4$ and intersection array $$(2^{2m},2^{2m}-1,2^{2m-1},1;1,2^{2m-1},2^{2m}-1,2^{2m}).$$
\end{itemize}
\end{thmc}

In fact, part (b) above is a consequence of a result of Borges, Rif{\`a}, and  Zinoviev \cite{BRZ}. Completely regular and completely transitive codes 
have been studied extensively, and the existence and enumeration of such codes are open hard problems (see \cite{distreg, delsarte, neum}
and more recently \cite{nonantipodal, BRZ, rho=2,giudici, binctrarb,kronprod,lifting}).

In \cite[Question 8.4]{Puzz} we ask the following question. Suppose a Conway groupoid associated to a supersimple design $\mathcal{D}$ is not
full and has a primitive hole stabilizer. Then does the incidence matrix of $\mathcal{D}$ generate a 
completely regular and/or uniformly packed $\mathbb{F}_p$-linear code for some prime $p > 0$? 
Since completely transitive codes are necessarily completely regular, by combining Theorems B and C we obtain an 
affirmative answer to this question for the designs $\De^\ep$, $\De^a$.

The remainder of the paper is concerned with (abstract) Conway groupoids and our next main result classifies all Conway groupoids that satisfy a particular group-theoretic condition. 
For a supersimple design $\De$ with point set $\Omega$, the hole stabilizer $G:=\pi_\infty(\De)$ is generated by elements of the form $[\infty, a, b, \infty]$
for $a,b\in\Omega\backslash\{\infty\}$ (see Section \ref{sub:design} for full discussion on $\pi_\infty(\De)$). The next result is dependent on the Classification of Finite Simple Groups (CFSG) through its use of Theorem~\ref{t: ls}.

\begin{thmd}
Suppose that $\De$ is a supersimple $2-(n,4,\lambda)$ design, that $\infty$ is a point in $\De$, and write $\L:=\L_\infty(\De)$. Suppose, furthermore, that $[\infty, a, b, \infty]=1$ whenever $\infty$ is collinear with $\{a,b\}$. Then one of the following is true:
\begin{enumerate}
 \item $\De$ is a Boolean design and $\L=(\mathbb{F}_2)^k$ for some $k > 0$;
 \item $\De=\mathbb{P}_3$ (the projective plane of order $3$) and $\L = M_{13}$; or 
 \item $\L = \Alt(n)$.
\end{enumerate}
\end{thmd}

Theorem D is a generalization of \cite[Theorem B]{Puzz} (concerning designs associated with trivial hole stabilizer) as well as a generalization of the classification of Conway groupoids associated with supersimple $2-(n,4,1)$ designs (when $\lambda=1$ the extra supposition is automatically satisfied). 

Theorem D is closely connected to our final main result, Theorem E, below. Indeed we will use Theorem E (2) to prove Theorem D, and then will use Theorem D to prove Theorem E (4).

\begin{thme}
Suppose that $\De$ is a supersimple $2-(n,4,\lambda)$ design, that $\infty$ is a point in $\De$, and that $\L:=\L_\infty(\De)$. Let $G:=\pi_\infty(\De)$ be the hole stabilizer of $\infty$, considered as a permutation group via its natural embedding in $\Sym(n)$.
 \begin{enumerate}
  \item If $n>4\lambda+1$, then $G$ is transitive;
  \item if $n> 9\lambda+1$, then $G$ is primitive;
  \item if $n> 144\lambda^2+120\lambda+26$, then $\L$ contains $\Alt(n)$;
%\item (CFSG-dependent) if $n>9\lambda^2+12\lambda+5$, then $G$ contains $\Alt(n-1)$.
\item  If $n>9\lambda^2-12\lambda+5$, then one of the following holds:
\begin{enumerate}
\item $\L$ contains $\Alt(n)$;
\item $\lambda=1$, $\De=\mathbb{P}_3$ (the projective plane of order $3$), and $\L =  M_{13}$. 
\end{enumerate}
\end{enumerate}
\end{thme}

Note that only the fourth item of Theorem E is dependent on CFSG.  Note too that if $\L$ contains $\Alt(n)$ (as in part (3) and (4) of the theorem), then
\[
 \L=\left\{\begin{array}{ll}
           \Alt(n), & \textrm{ if $\lambda$ is odd;} \\
           \Sym(n), & \textrm{ if $\lambda$ is even.}
          \end{array}\right.
\]

\subsection{Classifying Conway groupoids}
Theorem E provides a powerful tool in the program to classify Conway groupoids for arbitrary $\lambda$ and $n$. Such a classification was completed in \cite{Puzz} for $\lambda\leq 2$ and in Section \ref{s: l3} we make some remarks about the case $\lambda=3$. What about the general case?

Firstly note that Theorem E has an immediate corollary:

\begin{Cor}\label{c: e}
Let $\lambda$ be a positive integer. There are a finite number of (isomorphism classes of) groupoids that crop up as Conway groupoids associated with a supersimple $2-(n,4,\lambda)$ design.
\end{Cor}

Corollary~\ref{c: e} makes an interesting companion to Theorem E which implies that if $\lambda$ is allowed to vary, then there are an infinite number of (isomorphism classes) of groupoids that crop up as Conway groupoids.

One might naturally ask whether the bounds in Theorem E can be substantially improved as this would be an obvious aid to a classification. Unfortunately the relative dearth of examples of Conway groupoids makes this question difficult to answer: the only infinite families of non-full Conway groupoids which have been constructed to this point are those associated to the Boolean designs (for which $n=2\lambda+2$ \cite{Puzz}) and the examples in Theorem C (for which $n<5\lambda$). The parameters in these examples are a long way from the bounds given in Theorem E suggesting, perhaps, that there is plenty of room for improvement.

In a different direction we note that both Theorem D and Theorem E (4) suggest that the Conway groupoid associated to $M_{13}$ is particularly special. Indeed we have another reason to think this might be the case. 

Suppose that $\C$ is a Conway groupoid associated with a design $\De$, and suppose further that the hole stabilizer $\pi_\infty(\De)$ is primitive. If $\De$ is not $\mathbb{P}_3$, the projective plane of order $3$ and $\C$ is not full then in all examples known so far,  $\L_\infty(\De)$ is a transitive subgroup of $\Sym(n)$ with $\pi_\infty(\De)$ the stabilizer of the point $\infty$ in $\L_\infty(\De)$. Since, by supposition, $\pi_\infty(\De)$ is primitive, this implies that $\L_\infty(\De)$ is a {\bf 2-primitive} permutation group (i.e. a primitive group with a stabilizer primitive on its non-trivial orbit). We conjecture that this behaviour is general.

%We have a number of examples of primitive Conway groupoids (the two infinite family of Theorem D, along with the small examples listed in \cite{Puzz}; for the purposes of this discussion we may also include the conjectured examples discussed in Section~\ref{s:newinf}) and in all of these examples we have the remarkable fact that $\L_\De$ coincides as a set with a subgroup of $\Sym(n)$. For example, for the family of Theorem D, $\L_\De$ is equal as a set to the group $\Sp_{2m}(2)$ mentioned in the statement of the theorem.

\begin{Conj}\label{c: groupoid}
Suppose that $\De$ is a supersimple $2-(n,4,\lambda)$ design other than $\mathbb{P}_3$, that $\infty$ is a point in $\De$ and that the hole stabilizer $\pi_\infty(\De)$ is primitive. Then $\L_\infty(\De)$ coincides with a 
2-primitive subgroup $H$ of $\Sym(n)$ and $\pi_\infty(\De)$ is equal to the stabilizer in $\L_\infty(\De)$ of the point $\infty$. 
\end{Conj}

We remark that all 2-primitive permutation groups are known thanks to CFSG and the list is rather short (see \cite{mps} for some discussion). Thus this conjecture implies a very strong restriction on the structure of a Conway groupoid with primitive hole stabilizer and a proof would be a very significant step towards a classification.

One could push the conjecture a little further. Let us operate under the suppositions of Conjecture~\ref{c: groupoid} and assume, moreover, that $\L_\infty(\De)$ does not contain $\Alt(n)$. Now all known examples satisfy two further properties: 

Firstly, the elements of $\L_\infty(\De)$ are automorphisms of the design $\De$. Secondly, the group $\L_\infty(\De)$ is a {\it 3-transposition group} with associated class of transpositions coinciding with the set
\[
\{[a,b] \mid a,b\in \Omega\}.
\]
(Here $\Omega$ is the point set of $\De$; the elements $[a,b]$ are defined in Section \ref{sub:design}.)

In forthcoming work the authors prove Conjecture~\ref{c: groupoid} and the two additional statements just mentioned, provided the design $\De$ satisfies two mild combinatorial suppositions. It is expected that, by exploiting Fischer's famous theorem on 3-transposition groups, this will lead to a full classification of Conway groupoids in this restricted situtation \cite{Puzz2}.

\subsection{Structure of the paper}

The paper is structured as follows. Section \ref{s:back} provides the necessary background from design theory, group theory and coding theory.  In Section \ref{s:action} we give a precise description of the action of $\Sp_{2m}(2)$ on quadratic forms, introduce the designs $\De^a$ and $\De^{\ep}$ and prove Theorem A.
The Conway groupoids $\C(\De^a)$ and $\C(\De^\ep)$ are studied in Section \ref{s:puzz} where we establish Theorem B. In Section \ref{s:codes} we study the codes $C_{\mathbb{F}_2}(\De^a)$ and $C_{\mathbb{F}_2}(\De^\ep)$ in detail and give a  proof of Theorem C. 

Sections \ref{s:puzzlarge} and \ref{s:imprim} are devoted to the study of general Conway groupoids; in particular in  Section \ref{s:puzzlarge} we prove Theorem D  before proving Theorem E in Section \ref{s:imprim}. Section~\ref{s: l3} contains a discussion of the classification of Conway groupoids with $\lambda=3$.

\section{Background}\label{s:back}
\subsection{Block designs and moves}\label{sub:design}

%JASON NEEDS TO START FROM HERE

Recall that a \textit{balanced incomplete block design} $(\Omega,\B)$, or $t-(n,k,\lambda)$ design, is a finite set $\Omega$ of size $n$, together with a finite multiset $\B$ of subsets of $\Omega$ each of size $k$ (called \textit{lines}), such that any subset of $\Omega$ of size $t$ is contained in exactly $\lambda$ lines. 

In what follows $\De=(\Omega,\B)$ is a $2-(n,4,\lambda)$ design. We assume, moreover, that $\De$ is \textit{supersimple}, i.e. any pair of lines intersect in at most two points. (Note that, in particular, $\De$ is {\it simple}, i.e. there are no repeated lines.)

Let $a$ and $b$ be distinct points in $\Omega$. We define, first,
\begin{equation}\label{d:overline}\overline{a,b}:=\{x\in\Omega\,|\,\textnormal{there exists $\ell\in\B$ such that $x,a,b,\in\ell$}\}
\end{equation}
In particular, note that $a,b\in\overline{a,b}$.  

Next, we define the \textit{elementary move} associated with $a$ and $b$: this is the permutation 
\begin{equation}\label{e:elem move}
[a,b]:=(a,b)\prod_{i=1}^\lambda (a_i,b_i) \in \Sym(\Omega),
\end{equation}
 where $\{a,b,a_i,b_i\}$ is a line for each $1 \leq i \leq \lambda$. The fact that $\De$ is supersimple guarantees that the product \eqref{e:elem move} is well-defined. We note, moreover, that $[a,b]=[b,a]$ and that the set of points in $\Omega$ moved by the permutation $[a,b]$ (also called the \textit{support} of $[a,b]$) is precisely the set $\overline{a,b}$.
 
 A \textit{move sequence} is a product of elementary moves 
 \begin{equation}
[a_0,a_1,\ldots,a_k]:=[a_0,a_1] \cdot [a_1,a_2] \cdots [a_{k-1},a_k]  
 \end{equation}
where $a_{i-1},a_i \in \Omega$ for each $1 \leq i \leq k$. A move sequence $[a_0,a_1,\ldots,a_k]$ is called \textit{closed} if $a_0=a_k$.  
 
Suppose that $\infty$ is a point in $\Omega$. In this paper we will primarily study the following three sets for various designs $\De$:
 \begin{align}
 \L(\De) &:=\{[a_0,a_1,\ldots,a_k] \mid a_{i-1}, a_i\in\Omega \mbox{ for } 0 \leq i \leq k.\} \\
 \L_\infty(\De) &:=\{[\infty,a_1,\ldots,a_k] \mid a_{i-1}, a_i\in\Omega \mbox{ for } 1 \leq i \leq k.\} \\
 \pi_\infty(\De) &:=\{[\infty,a_1,\ldots,a_{k-1},\infty] \mid a_{i-1}, a_i\in\Omega \mbox{ for } 1 \leq i \leq k-1.\} 
 \end{align}
Observe that $\pi_\infty(\De)\subseteq \L_\infty(\De)\subseteq \L(\De)$. 

The set $\pi_\infty(\De)$ is called the {\it hole stabilizer}; it is precisely the set of all closed move sequences which start and end at $\infty$. It is an easy exercise to confirm that $\pi_\infty(\De)$ is a group. We recall that $\pi_\infty(\De)$ is generated by elements of the form $[\infty, a, b, \infty]$
for $a,b\in\Omega\backslash\{\infty\}$ \cite[Lemma 3.1]{Puzz} and that if $\infty_1$ and $\infty_2$ are distinct elements of $\Omega$, then $\pi_{\infty_1}(\De) \cong\pi_{\infty_2}(\De)$ \cite[Theorem A]{Puzz}, since the hole stabilizers are conjugate subgroups of $\Sym(n)$.

By way of example, note that if $\De=\mathbb{P}_3$, the projective plane of order $3$, then $\L_\infty(\De)$ is equal to the set $M_{13}$ originally defined by Conway. The group $\pi_{\infty}(\De)$ is, then, a subset of $\Sym(12)$ isomorphic to the Mathieu group $M_{12}$.

\subsection{Conway groupoids}\label{sub:conway}

Let $\De=(\Omega,\B)$ be a supersimple $2-(n,4,\lambda)$ design, as before. 
 The \textit{Conway groupoid} $\C(\De)$ is the small category whose object set is $\Omega$ and such that, for $a,b\in \Omega,$ 
 $$\Mor(a,b) :=\{[a,a_1,\ldots,a_{k-1},b] \mid a_{i-1}, a_i\in\Omega \mbox{ for } 1 \leq i \leq k-1.\} .$$ 
Observe that the set of all morphisms in the category $\C(\De)$ is equal to the set $\L(\De)$. 

Two Conway groupoids $\C(\De_1)$ and $\C(\De_2)$ are {\it isomorphic} if they are isomorphic as categories, i.e. there exist two mutually inverse functors between $\C(\De_1)$ and $\C(\De_2)$. It is easy to check that this condition is equivalent to the condition that $\De_1$ and $\De_2$ contain the same number of points, $n$, and, moreover, that there exists $\phi\in\Sym(n)$ such that 
\[\L(\De_2)=\left(\L(\De_1)\right)^\phi:= \{ \phi^{-1}g\phi \mid g\in \L(\De_1)\}.\]

We return to the design $\De$ and fix a point $\infty$ in $\Omega$. Clearly, for each $a,b\in\Omega$ and each $\sigma \in \Mor(a,b)$, there exist $\rho,\tau \in \L_\infty(\De)$ such that $\sigma= \rho \cdot \tau^{-1}$. In particular the category $\C(\De)$ is completely determined by $\L_\infty(\De)$. 

This straightforward observation underpins our work from here on: rather than studying the category $\C(\De)$ directly, we prefer to study the set $\L_\infty(\De)$. Indeed, in earlier literature on this subject these two objects have been treated as somewhat interchangeable: the label $M_{13}$, for instance, is sometimes used to refer to the set $\L_\infty(\mathbb{P}_3)$, sometimes to the groupoid $\C(\mathbb{P}_3)$. In what follows we will always treat $M_{13}$ as a set and, indeed, we will have no need to study $\C(\De)$ for any design at all.

%one reads of ``the groupoid $M_{13}$'', despite the fact that as defined $M_{13}$ is a subset of $\Sym(13)$ rather than a category. the label $M_{13}$ which we use in this paper for a particular subset of $\Sym(13)$ is also determThis line of inquThis approach hearkens back to the original approach of Conway who n order to study Conway groupoids, it is sufficient to study the sets $\L_\infty(\De)$ associated with designs $\De$; hence the set $\L_\infty(\De)$ will be our primary focus This 

%Hence we think of a $2-(n,k,\lambda)$ design as an $(n,k)$-line system such that any $2$-subset of $\Omega$ is contained in $\lambda$ lines.

\subsection{Permutation groups}
\label{sub:permutation}

Let $G$ be a finite group acting on a non-empty set $\Omega$. The action is \emph{transitive} if for any $x, y \in \Omega$ there exists $g \in G$ such that $x^g = y$ and \emph{$t$-transitive} if the induced action on the set of all $t$-tuples of distinct elements of $\Omega$ is transitive for some $t > 0$.

Suppose that the action of $G$ on $\Omega$ is transitive. A \textit{system of imprimitivity} is a partition of $\Omega$ into $\ell$ subsets $\Delta_1,\Delta_2,\ldots,\Delta_\ell$ each of size $k$ such that $1< k,\ell< n$, and so that for all $i\in\{1,\dots,\ell\}$ and all $g\in G$, there exists $j\in\{1,\dots,\ell\}$ such that
\[
 \Delta_i^g=\Delta_j.
\]
The sets $\Delta_i$ are called \textit{blocks}. We say that $G$ acts \textit{imprimitively} if there exists a system of imprimitivity. If no such set exists then $G$ acts \textit{primitively} on $\Omega$.

\subsection{Linear Codes}
Let $C$ be a linear binary code of length $n$, i.e. $C$ is a subspace of the vector space $(\mathbb{F}_2)^n$. Recall that elements of $C$ are called {\it codewords}. 

We define the \emph{binary Hamming graph} $\Gamma=H(n,2)$ to be the finite graph with vertex set $V(\Gamma)=(\mathbb{F}_2)^n$, such that an edge exists between two vertices if and only if they differ in precisely one entry. Observe that $C$ is a subset of the vertex set of $\Gamma$.

For all pairs of vertices $\alpha,\beta\in V(\Gamma)$, the \emph{Hamming distance} between
$\alpha$ and $\beta$, denoted by $d(\alpha,\beta)$, is defined to be
the number of entries in which the two vertices differ.  We let
$\Gamma_k(\alpha)$ denote the set of vertices in $H(n,2)$ that are at
distance $k$ from $\alpha$.

We are now able to define the \emph{minimum distance, $d$,
  of C} to be the smallest distance between distinct codewords of $C$.
For any $\gamma\in V(\Gamma)$, we define 
$$d(\gamma,C)=\min\{d(\gamma,\beta)\,:\,\beta\in C\}$$ to be the 
\emph{distance of $\gamma$ from $C$}.  The \emph{covering radius of $C$}, 
which we denote by $\rho$, is the maximum distance that any vertex in $H(n,2)$ is from $C$.  
We let $C_i$ denote the set of vertices that are at distance $i$ from $C$; %, and deduce,
%for $i\leq \lfloor (\delta-1)/2\rfloor$, that $C_i$ is the disjoint 
%union of $\Gamma_i(\alpha)$ as $\alpha$ varies over $C$.  
then $C_0=C$ and $\{C,C_1,\ldots,C_\rho\}$ forms a
partition of $V(\Gamma)$ called the \emph{distance partition of $C$}. For each $i$, the set $C_i$ is a union of cosets of $C$, and we say that a coset that is a subset of $C_i$ is {\it of weight $i$}.

The automorphism group of $\Gamma$, 
$\Aut(\Gamma)$, is the semi-direct product
$B\rtimes L$ where $B\cong \Sym(2)^n$ and $L\cong \Sym(n)$, see \cite[Theorem 9.2.1]{distreg}.    Let
$g=(g_1,\ldots, g_n)\in B$, $\sigma\in L$ and
$\alpha=(\alpha_1,\ldots,\alpha_n)\in V(\Gamma)$. Then $g$ and $\sigma$
act on $\alpha$ in the following way: \begin{equation}\label{autgpaction}
\alpha^g=(\alpha_1^{g_1},\ldots,\alpha_n^{g_n}),\,\,\,\,\,\,\,\alpha^\sigma=(\alpha_{1\sigma^{-1}},\ldots,\alpha_{n\sigma^{-1}}). \end{equation}

The \emph{automorphism group of $C$}, denoted by $\Aut(C)$, is the setwise stabiliser of $C$ in $\Aut(\Gamma)$.
In this paper, we construct a family of codes with the following symmetrical property.  

\begin{Def}\label{d:ctrans} Let $C$ be a code with distance partition
  $\{C=C_0, C_1,\ldots,C_\rho\}$.  We say $C$ is \emph{$X$-completely transitive}, or simply \emph{completely transitive}, if there exists $X\leq \Aut(\Gamma)$ such that
    $C_i$ is an $X$-orbit for $i=0,\ldots,\rho$.
\end{Def}

\noindent It is known that completely transitive codes are necessarily completely regular \cite{giudici}.

\begin{Def}\label{d:creg} A binary code $C$ with covering radius $\rho$ is \emph{completely regular} if for all $i\geq 0$, 
every vector $\alpha\in C_i$ has the same number $c_i$ of neighbours in $C_{i-1}$
and the same number $b_i$ of neighbours in $C_{i+1}$; note that $c_0=b_\rho=0$. For such a code, define $(b_0,\ldots,b_{\rho-1};c_1,\ldots,c_\rho)$
to be the \emph{intersection array} of $C$.
\end{Def}

Recall that the \textit{dimension} of $C$ is the dimension of $C$ regarded as a vector space over $\mathbb{F}_2$. We say that $C$ is an $[n,k,d]$ code if it has minimum distance $d$ and dimension $k$. We will need the following result from \cite{lifting}.

\begin{Lem}\label{l:rifa} Let $C$ be a linear completely regular $[n,k,d]$ code
with covering radius $\rho$ and intersection array $(b_0,\ldots,b_{\rho-1};c_1,\ldots,c_\rho)$. Let $\mu_i$ denote the number
of cosets of $C$ of weight $i$, where $i=0,\ldots,\rho$. Then the following equality holds:$$b_i\mu_i=c_{i+1}\mu_{i+1},\,\,i=0,\ldots,\rho-1.$$
\end{Lem}

\section{The actions of \texorpdfstring{$\Sp_{2m}(2)$ and $2^{2m}.\Sp_{2m}(2)$}{Sp(2m,2) and 2^2m.Sp(2m,2)} on quadratic forms}\label{s:action}
The notation and terminology in this section will be based on that found in \cite[Section 7.7]{Permutation}. We start with the standard construction for the action of $\Sp_{2m}(2)$ on quadratic forms. Let $m \geq 1$ be an integer and $V:=\mathbb{F}_2^{2m}$ be a vector space equipped with the standard basis and consider the block matrices \[e=\begin{pmatrix}0_m&I_m\\0_m&0_m\end{pmatrix}, \qquad
f=\begin{pmatrix}0_m&I_m\\I_m&0_m\end{pmatrix}=e+e^T.\]

We write elements of $V$ as row matrices and, therefore, define $\varphi(u,v)$ to be the alternating bilinear form given by $\varphi(u,v):=ufv^T$. We let $\Omega$ be the set of all quadratic forms $\theta(u)$ with the property that \[ \varphi(u,v)=\theta(u+v)+\theta(u)+\theta(v), \]
i.e. $\Omega$ is the set of quadratic forms whose polarisation is equal to $\varphi$. Now we define $\theta_0(u):=ueu^T \in \Omega$, and by results in \cite[Section 7.7]{Permutation}, any other element of $\Omega$ is of the form \[ \theta_a(u) := \theta_0(u)+\varphi(u,a), \]
where $a$ is a fixed element of $V$.

Recall that $\Sp_{2m}(2):=\{A \in \textrm{GL}_{2m}(2) \mid AfA^T=f \}$ acts on $\Omega$ (on the right) via $\theta^x(u):=\theta(ux^{-1})$ for each $\theta \in \Omega$ and $x \in \Sp_{2m}(2)$.
Recall (\cite[Corollary 7.7A]{Permutation}) that the action of $\Sp_{2m}(2)$ on $\Omega$ splits into two distinct orbits \[  \Omega^0:=\{\theta_a \mbox{ } | \mbox{ } a \in V^0\},\qquad \Omega^1:=\{\theta_a \mbox{ }|\mbox{ } a \in V^1\}\] 
where
\[  V^0:=\{a \in V \mbox{ }|\mbox{ } \theta_0(a)=0\},\qquad V^1:=\{a \in V \mbox{ }|\mbox{ } \theta_0(a)=1\}.\]

Given the form $\varphi$ and an element $c\in V$, we define the \textit{transvection} $t_c$ as follows:
$$u^{t_c}:=u + \varphi(u,c)c, \mbox{ for all } u,c \in V.$$ 
Recall that the set of all transvections generates $\Sp_{2m}(2)$ (see, for instance, \cite[Theorem 8.5]{Ta}). The following result is \cite[Lemma 7.7A]{Permutation}.

\begin{Lem}\label{l:trans}
The following hold:
\begin{itemize}
\item[(i)] For all $a,c \in V$, $$\theta_a^{t_c}=\left\{
	\begin{array}{ll}
		\theta_a,  & \mbox{if } \theta_a(c)=1; \\
		\theta_{a+c}, &  \mbox{if } \theta_a(c)=0 
	\end{array}
\right.$$

\item[(ii)] For each $a,b \in V$ there is at most one $c \in V$ such that $t_c$ maps $\theta_a$ onto $\theta_b$. Such a $c$ exists if and only if $\theta_0(a)=\theta_0(b)$ (and then $c=a+b$).
%\item[(iii)] $H = \langle t_c \mid c \in V \rangle$. 
\end{itemize}
\end{Lem}

%\begin{proof}
%Parts (i) and (ii) appear in } while part (iii) is shown in \cite{Ta}.
%\end{proof}

As an immediate consequence, we obtain:

\begin{Lem}\label{l:evensum}
Let $\epsilon \in \mathbb{F}_2$ and $\{v_1,\ldots,v_k\}$ a subset of $V^\epsilon$ for some odd integer $k > 0$. Then, for each $g \in \Sp_{2m}(2)$, we have 
\begin{equation} \sum_{i=1}^k (\theta_{v_i})^g = \left(\theta_{\sum_{i=1}^k v_i} \right)^g.
\end{equation}
\end{Lem}

\begin{proof}
We begin by considering the case where $g=1$. Since $k$ is odd,
\[\sum_{i=1}^k \theta_{v_i}(u)=\sum_{i=1}^k \theta_0(u) + \sum_{i=1}^k \varphi(u,v_i)= \theta_0(u) + \varphi(u,\sum_{i=1}^k v_i) = \theta_{\sum_{i=1}^k v_i}(u).\]
We now turn to the general case. Since the transvections generate $\Sp_{2m}(2)$, it suffices to consider the case $g=t_c$ for some $c \in V.$ We calculate,  
\begin{align*}\sum_{i=1}^k \theta_{v_i}(u)^{t_c}&=\sum_{i=1}^k \theta_{v_i+(1+\theta_{v_i}(c))c}(u) \\
&=\sum_{i=1}^k \theta_0(u)+\varphi(u,v_i+(1+\theta_{v_i}(c))c)  \\
&=\theta_0(u)+\varphi(u,\sum_{i=1}^k v_i+c+c\sum_{i=1}^k \theta_{v_i}(c)) \\
&= \theta_0(u)+\varphi(u,\sum_{i=1}^k v_i+(1+\theta_{\sum_{i=1}^k v_i}(c))c)  =  \left(\theta_{\sum_{i=1}^k v_i} \right)^{t_c}(u).
\end{align*}
\end{proof}

We now show how to decompose elements of $V$ into a sum of elements in $V^\ep$, which will prove useful in the sequel.

\begin{Lem}\label{l:sumof2}
For each $v \in V$ and $\epsilon \in \mathbb{F}_2$ there exist distinct $x,y \in V^\epsilon$ such that $v=x+y$.
\end{Lem}

\begin{proof}
We prove this in a series of cases. In each case, let $e_i$ denote the $i$'th basis vector for $1 \leq i \leq 2m$. Let $y=x+v$ and $\delta:=\theta_0(v)$. 
\begin{itemize}
\item[(a)] If $\delta=0$, $\epsilon=0$ let $x:=0$.
\item[(b)] If $\delta=0$, $\epsilon=1$ then
\begin{itemize}
\item[(i)] if $v_i=v_{i+m}$ for some $1 \leq i \leq m$, let $x:=e_i+e_{i+m}$; 
\item[(ii)] if $v_i \neq v_{i+m}$ for all $1 \leq i \leq m$, fix any $i$, let $j$ be such that either $v_{j-m}=1$ or $v_{j+m}=1$ and let $x:=e_i+e_{i+m}+e_j$.
\end{itemize}
\item[(c)] If $\delta=1$, $\epsilon=0$ let $1 \leq i \leq m$ be such that $v_i=v_{i+m}=1$ and let $x:=e_i$.
\item[(d)] If $\delta=1$, $\epsilon=1$ let $1 \leq i \leq m$ be such that $v_i=v_{i+m}=1$. Then
\begin{itemize}
\item[(i)] if $v_j=v_{j+m}$ for some $1 \leq j \leq m$ with $j \neq i$, let $x:=e_j+e_{j+m}+e_i$;
\item[(ii)] if $v_j \neq v_{j+m}$ for all $1 \leq j \leq m$ with $j \neq i$, let $j$ be such that $v_{j-m}=1$ or $v_{j+m}=1$ and let $x:=e_i+e_{i+m}+e_j$.
\end{itemize}
\end{itemize}
\end{proof}

\begin{Cor}\label{c:sumof2} Let $v\in V^{\ep}$. Then $v$ can be written as the sum of three distinct elements of $V^{1-\ep}$.
\end{Cor}

\begin{proof} By Lemma \ref{l:sumof2}, $v=x+y$ for some $x,y\in V^{1-\ep}$.
Again, by Lemma \ref{l:sumof2}, $y=y_1+y_2$ for some $y_1,y_2\in V^{1-\ep}$ and so $v=x+y_1+y_2$.  Now if any of $x,y_1,y_2$ are equal, then $v\in V^{1-\ep}$, which is a contradiction.  
\end{proof}

\subsection{The action of \texorpdfstring{$\Sp_{2m}(2)$}{Sp(2m,2)} on \texorpdfstring{$3$-subsets}{3-subsets}}
In \cite[Theorem 7.7A]{Permutation}, the authors deduce that $\Sp_{2m}(2)$ acts 2-transitively on $\Omega^\epsilon$ for $\epsilon \in \mathbb{F}_2$. In fact, more is true:

\begin{Thm}\label{t:3sets}
Let $\epsilon, \delta \in \mathbb{F}_2$ and $m \geq 3$. The action of $\Sp_{2m}(2)$ on $3$-subsets of elements in $\Omega^\epsilon$ splits into two orbits, $\O_0^\epsilon$ and $\O_1^\epsilon$, defined as follows: 
\[ \O_\delta^\epsilon:=\Big\{\{\theta_{v_1}, \theta_{v_2}, \theta_{v_3}\} \mid v_j \in V^\epsilon, \theta_0(v_1+v_2+v_3)=\delta\Big\}. \]
Furthermore, for each $v \in V^\delta$, the sets 
\[\Delta^\epsilon_v:=\left\{\{\theta_{v_1}, \theta_{v_2}, \theta_{v_3}\} \in \O_\delta^\epsilon \mid \sum_{i=1}^3 v_i = v \right\}
\]
 form blocks of imprimitivity for the action  of $\Sp_{2m}(2)$ on $\O_\delta^\epsilon$.
\end{Thm}

We will prove Theorem~\ref{t:3sets} shortly. In order to do so we need a definition from \cite{Permutation}: Let $a\in V, \epsilon \in \mathbb{F}_2$ and set $$L(a,\epsilon):=\{v \in V \mid \varphi(v,a)=\epsilon\}.$$ 
Observe that $L(a,0)$ is a subspace of $V$ for all $a\in V$.
Before the proof of Lemma 7.7B in \cite{Permutation}, it is shown that $$\dim \left( \bigcap_{i=1}^k L(a_i,0)\right)=2m-k,$$ whenever $\{a_1,\ldots a_k\}$ are linearly independent. The following is a generalisation of \cite[Lemma 7.7B]{Permutation}.

\begin{Lem}\label{l:linind}
Let $m \geq 4$ and let $a,b,c$ be linearly independent vectors in $V$. For any $\epsilon_1,\epsilon_2,\epsilon_3 \in \mathbb{F}_2$, $\theta_0$ is not constant on
$$L(a,\epsilon_1) \cap L(b,\epsilon_2) \cap L(c,\epsilon_3). $$
\end{Lem}

\begin{proof}
By assumption, $U:=L(a,0) \cap L(b,0) \cap  L(c,0)$ is a subspace of dimension $2m-3 > 3$ in $V$, so there is $d \in U$ which is linearly independent of $a,b$ and $c$. This means we may choose $$w \in L(a,\epsilon_1) \cap L(b,\epsilon_2) \cap L(c,\epsilon_3) \cap L(d,\epsilon_4) $$ for any $\epsilon_4 \in \mathbb{F}_2$. 
The fact that $d\in U$ implies that $w,w+d \in L(a,\epsilon_1) \cap L(b,\epsilon_2) \cap L(c,\epsilon_3)$, so that on setting $\epsilon_4:=\theta_0(d)+1$, we have: $$\theta_0(w+d)=\theta_0(w)+ \theta_0(d)+\varphi(w,d)=\theta_0(w)+1,$$ as needed.
\end{proof}

We can now prove Theorem \ref{t:3sets}:

\begin{proof}
Fix $\epsilon,\delta \in \mathbb{F}_2$. We first prove that $\O_\delta^\epsilon$, $\O_{1-\delta}^\epsilon$ are the two distinct orbits of $\Sp_{2m}(2)$ on 3-subsets of $\Omega^\epsilon$.
By Lemma \ref{l:evensum}, both $\mathcal{O}^\ep_\delta$ and $\O_{1-\delta}^\epsilon$ are fixed setwise by $\Sp_{2m}(2)$.
When $m=3$, a GAP \cite{GAP} calculation verifies that each is in fact an $\Sp_{2m}(2)$-orbit.
Thus we assume from now on that $m \geq 4$.

Since $\Sp_{2m}(2)$ acts 2-transitively on $\Omega^\epsilon$, it is sufficient to prove that whenever $a,b,c,d \in V^\ep$,
there is an element $x$ of $\Sp_{2m}(2)$ which fixes $\theta_c,\theta_d$ but maps $\theta_a$ to $\theta_b$ if and only if $\theta_0(a+c+d)=\theta_0(b+c+d)$. 
(Recall that $\{\theta_a,\theta_c,\theta_d\}$ and $\{\theta_b,\theta_c,\theta_d\}$ are both elements in $\mathcal{O}^\ep_\delta$, or both elements in
 $\O_{1-\delta}^\epsilon$, if and
only if $\theta_0(a+c+d)=\theta_0(b+c+d)$.)

In order to prove this fact, we will show that there is $w \in V^\epsilon$ such that
\begin{equation}\label{3transcond}
\theta_c(a+w)=\theta_c(b+w)=\theta_d(a+w)=\theta_d(b+w)=1
\end{equation} if and only if $\theta_0(a+c+d)=\theta_0(b+c+d)$. Note that, since $w\in V^\epsilon$, we easily deduce that
\[
 \theta_a(a+w)=\theta_w(b+w)=0.
\]
This, along with \eqref{3transcond} and Lemma \ref{l:trans}, implies that we may take $x=t_{a+w} \cdot t_{b+w}$ and we are done.

Thus it remains to show that there is $w\in V^\epsilon$ satisfying \eqref{3transcond}. One easily checks that \eqref{3transcond} is equivalent to
\begin{align*}
\varphi(w,a+c)&=1+\varphi(a,c);  \\
\varphi(w,a+d)&=1+\varphi(a,d); \\
\varphi(w,b+c)&=1+\varphi(b,c); \\
\varphi(w,b+d)&=1+\varphi(b,d). 
\end{align*}

Since the vectors $\{a+c,b+c,a+d\}$ are linearly independent, Lemma \ref{l:linind}  implies that $\theta_0$ is not constant on $$L(a+c,1+\varphi(a,c)) \cap L(b+c,1+\varphi(b,c)) \cap L(a+d,1+\varphi(a,d))$$ (notice that this assertion holds even if $d=a+b+c$ by \cite[Lemma 7.7B]{Permutation}.) Thus whatever value $\epsilon$ takes, there exists $w\in V^\ep$ satisfying the conditions in (\ref{3transcond}) if and only if $\varphi(w,b+d)=1+\varphi(b,d)$ holds above. But $\theta_0(a+c+d)=\theta_0(b+c+d)$ if and only if $\varphi(b,c)+\varphi(b,d)=\varphi(a,c)+\varphi(a,d)$ which is if and only if $$\varphi(w,b+d)=\varphi(w,b)+\varphi(w,d)=\varphi(b,c)+\varphi(w,c)+\varphi(a,d)+\varphi(w,a)$$
$$=\varphi(b,d)+\varphi(a,c)+\varphi(w,c)+\varphi(w,a)=1+\varphi(b,d),$$
as required. This proves the first assertion in Theorem \ref{t:3sets}.

It remains to prove the last statement. Let $v \in V^\delta$, $\{\theta_{v_1},\theta_{v_2},\theta_{v_3}\} \in \Delta_v^\epsilon$ and $c \in V$. By Lemma \ref{l:evensum}, $$\theta_{v_1}^{t_c}+\theta_{v_2}^{t_c}+\theta_{v_3}^{t_c}=\theta_{v_1+v_2+v_3}^{t_c},$$
so that either $(\Delta_v^\epsilon)^{t_c} \cap \Delta_v^\epsilon = \emptyset$ or $(\Delta_v^\epsilon)^{t_c} \cap \Delta_w^\epsilon = \Delta_v^\epsilon$ according to whether $\theta_v(c)=0$ or $1$ respectively. Since $\Sp_{2m}(2)$ is transitive on $\O_\delta^\epsilon$,  $\Delta:=\{\Delta_v^\epsilon \mid v \in V^\delta\}$ forms a system of imprimitivity for the action of $\Sp_{2m}(2)$ on $\O_\delta^\epsilon$ (in fact the action of $\Sp_{2m}(2)$ on $\Delta$ is equivalent to its action on $\Omega^\delta$) and the proof is complete.
\end{proof}

\subsection{The action of \texorpdfstring{$2^{2m}.\Sp_{2m}(2)$}{Sp(2m,2)} on \texorpdfstring{$3$-subsets}{3-subsets}}
We next consider an analogous situation for the affine group $2^{2m}.\Sp_{2m}(2)$ whose elements may be identified with pairs $(v,g)$ with $v \in (\mathbb{F}_2)^{2m}$ and $g \in \Sp_{2m}(2)$, and the action on $V$ is given by 
\begin{equation}\label{e:action}
u^{(v,g)}=u^g+v^g\,\,\,\forall u\in V.
\end{equation}

\begin{Thm}\label{t:3sets2}
Let $m \geq 3$. The action of $2^{2m}.\Sp_{2m}(2)$ on $3$-subsets of elements in $V$ splits into two orbits, $\O_0$ and $\O_1$, defined as follows:

$$\O_\delta:=\{\{v_1,v_2,v_3\} \mid v_i \in V, \varphi(v_1,v_2)+\varphi(v_1,v_3)+\varphi(v_2,v_3)=\delta\}.$$
%Furthermore, for each $v \in V$, the sets $$\Delta_v:=\{\{v_1,v_2,v_3\} \in \O_\delta \mid \sum_{i=1}^3 v_i=v\}$$ form blocks of imprimitivity for the action of $H$ on $\O_\delta.$
\end{Thm}

\begin{proof}
When $m=2,3$ we verify all assertions via a  GAP \cite{GAP} computation, so we assume from now on that $m \geq 4$. 
Let $\{v_1,v_2,v_3\}\in\mathcal{O}_\delta$.  Then a straight forward calculation shows that
$$\sum^3_{\text{i,j=1}\atop\text{i<j}}\varphi(v_i^{(v,g)},v_j^{(v,g)})=\sum^3_{\text{i,j=1}\atop\text{i<j}}\varphi(v_i^g,v_j^g),$$
and because $\Sp_{2m}(2)$ preserves $\varphi$, the right hand side of the above equation is equal to $\delta$. Hence $\mathcal{O}_\delta$
is fixed setwise by $2^{2m}.\Sp_{2m}(2)$.  

As is well-known, $2^{2m}.\Sp_{2m}(2)$ acts $2$-transitively on $V$, so it suffices to show that for each $a,b,c,d \in V$ there is an element of $2^{2m}.\Sp_{2m}(2)$ which fixes $c,d$ and maps $a$ to $b$ if and only if \begin{equation}\label{eq3} \varphi(a,c)+\varphi(a,d)=\varphi(b,c)+\varphi(b,d).
\end{equation} As in the proof of Theorem \ref{t:3sets} there exists $w \in V$ such that $$\varphi(w,a+c)=1+\varphi(a,c);$$  
$$\varphi(w,a+d)=1+\varphi(a,d);  $$
$$\varphi(w,b+c)=1+\varphi(b,c).  $$

Summing up the left hand sides of these equations and using (\ref{eq3}) we see that $\varphi(w,b+d) = 1+\varphi(b,d)$. It is now easy to verify that the element $g \cdot h \in 2^{2m}.\Sp_{2m}(2)$ has the required property with $$g:=(w^{t_a}+a^{t_w},t_{a+w}) \mbox{ and } h:=(w^{t_b}+b^{t_w},t_{b+w}).$$ 

\end{proof}

\subsection{Construction of \texorpdfstring{$\De^\ep$}{D+-}}
For $\ep \in \mathbb{F}_2$, it follows from Lemma \ref{l:evensum} that we can describe $\B^\ep$ from Theorem A as:
$$\B^\ep:=\{\{\theta_a,\theta_b,\theta_c,\theta_{a+b+c}\} \mid \{\theta_a,\theta_b,\theta_c\} \in \O_\ep^\ep \}.$$

%$$\B^\ep:=\{\{\theta_1,\theta_2,\theta_3,\theta_4\} \mid \{\theta_1,\theta_2,\theta_3\} \in \O_\ep^\ep \mbox{ and } \theta_1+\theta_2+\theta_3=\theta_4\}.$$

\begin{Lem}\label{l:spdes}
$\De^\ep:=(\Omega^\ep,\B^\ep)$ is a supersimple $2-(|\Omega^\ep|,4,\lambda^\ep)$ design for some $\lambda^\ep > 0$.
\end{Lem}

\begin{proof}
Clearly $\B^\ep$ contains no repeated lines (by definition). Moreover, given any $\ell\in \B^\ep$, any three points in $\ell$
uniquely determine the fourth, so the intersection of any two lines has size at most 2. As $\O_\ep^\ep$ is an $\Sp_{2m}(2)$-orbit, 
we deduce from Lemma \ref{l:evensum} that $\B^\ep$ is a $\Sp_{2m}(2)$-orbit on the 4-subsets of $\Omega^\ep$.  Since $\Sp_{2m}(2)$ acts 2-transitively on $\Omega^\ep$, $\De^\ep:=(\Omega^\ep,\B^\ep)$ is a $2-(|\Omega^\ep|,4,\lambda^\ep)$ design for some $\lambda^\ep > 0$ by \cite[Lemma 4.3]{Puzz}.
\end{proof}

%I THINK THIS IS NOW VOID!!
%We  record the  formula for $|\O^\ep_\delta|:$

%\begin{Lem}\label{l:sizeoforbs} $m \geq 3$, $n:=|V^\ep|$ and $\mathcal{O}_\delta^{\ep}$ be as defined in Theorem \ref{t:3sets}.  Then
%$$|O^{\ep}_\ep|=\frac{n(n-1)\lambda_\ep}{3}, \mbox{ and } |\mathcal{O}^{\ep}_{1-\ep}|=\frac{n(n-1)(n-2(\lambda_{\ep}+1))}{6}.$$

%\end{Lem}

%\begin{proof}
%Elementary counting yields: $$|\mathcal{O}^\ep_\ep|=4|\B^\ep|=\frac{n^\ep(n^\ep-1)\lambda^\ep}{3},$$ and counting subsets of size 3, we have $$|\mathcal{O}_\ep^{1-\ep}|=\binom{n^\ep}{3}-\frac{n^\ep(n^\ep-1)\lambda^\ep}{3}=\frac{n^\ep(n^\ep-1)(n^\ep-2(\lambda^\ep+1))}{6},$$ by Theorem \ref{t:3sets}.
%\end{proof}

For the design $\mathcal{D}^\ep$, it remains to calculate the values of $|\Omega^\ep|$ and $\lambda^\ep$. It is well known that $n^\ep:=|\Omega^\ep|=|V^\ep|=2^{m-1}(2^m + (-1)^\epsilon)$. One proof of this comes from a (probably well known) inductive construction for $V^\epsilon$, which we now describe.

%Write $V^\epsilon:=V_m^\epsilon$ so that the dimension is obvious. For each $x,y \in \mathbb{F}_2$ and $k > 0$, let $(V_k^\epsilon)^{xy}:=\{v_{xy} \mid v \in V_k^\epsilon \} \subseteq V_{k+1}$ where for each $v \in V_k^\epsilon$, $v_{xy}$ is the vector formed from $v$ by first splitting $v$ into two components of length $k$, appending $x$ to the first component and $y$ to the second component and then reamalgamating the resulting vectors to obtain a vector of length $2k+2$.   

For $k>0$ let $V_k$ denote the $\mathbb{F}_2$-vector space of dimension $2k$, and as before, $V^\ep_k=\{v\in V_k\,|\,\theta_0(v)=\ep\}$
where $\theta_0$ is defined over the appropriate dimension. For each $x,y \in \mathbb{F}_2$, $k>0$ and $v=(v_1,v_2)\in V_k$ (here each $v_i$ is an $\mathbb{F}_2$-vector of length $k$), let $v_{xy}=(x,v_1,y,v_2)\in V_{k+1}$. Moreover, let $(V_k^\epsilon)^{xy}:=\{v_{xy} \mid v \in V_k^\epsilon \} \subseteq V_{k+1}$.
\begin{Lem}\label{l:indcons}
For each $\epsilon \in \mathbb{F}_2$, $$V_{k+1}^\epsilon=(V_k^\epsilon)^{00} \cup (V_k^\epsilon)^{01} \cup (V_k^\epsilon)^{10} \cup (V_k^{1-\epsilon})^{11}.$$ In particular, $|V_{k+1}^\epsilon|=3|V_k^\epsilon|+|V_k^{1-\epsilon}|$ and $|V_m^\epsilon|=2^{m-1}(2^m + (-1)^\epsilon)$.
\end{Lem}

\begin{proof}
Clearly each of $(V_k^\epsilon)^{00}$, $(V_k^\epsilon)^{01}$, $(V_k^\epsilon)^{10}$ and $(V_k^{1-\epsilon})^{11}$ is contained in $V_{k+1}^\epsilon$. Conversely any element of $V_{k+1}^\epsilon$ must lie in one of these sets. Thus, since these sets are pairwise disjoint, $|V_{k+1}^\epsilon|=3|V_k^\epsilon|+|V_k^{1-\epsilon}|$. Obtaining an explicit formula for $|V_k^\epsilon|$ is now safely left as an exercise.
\end{proof}

%We now compute $\lambda^\epsilon$ inductively. 

%\begin{Lem}\label{l:indcons2}
%Write $\lambda^\epsilon:=\lambda_m^\ep$. Then $\lambda_{k+1}^\epsilon=3(\lambda^\epsilon_k+1)+\lambda_k^{1-\epsilon}$ and $$\lambda_m^\epsilon=2^{m-2}(2^{m-1}+(-1)^{\epsilon})-1.$$
%\end{Lem}

%\begin{proof} 
%We first deduce the relation $$\lambda_{k+1}^\epsilon=3(\lambda^\epsilon_k+1)+\lambda_k^{1-\epsilon}.$$ Observe that $\lambda_k^\epsilon$ is exactly the quantity $$|\{v \in V_k \mid \theta_0(v_1+v_2+v)=\epsilon\}|$$ for any $v_1,v_2 \in V_k^\epsilon$. Since $H$ acts $2$-transitively on $\Omega^\epsilon$, we can assume that either

%\begin{itemize}
%\item[(i)] $v_1=0$ and $v_2=e_1$, if $\epsilon=0$; or
%\item[(ii)] $v_1=e_1+e_2+e_{m+2}$ and $v_2=e_2+e_{m+2}$, if $\epsilon=1$.
%\end{itemize}
%Thus in either case, $v_1+v_2=e_1$, and it is elementary to count: $$2+2\lambda_k^\epsilon=|\{v \in V \mid v \cdot e_{m+1} = \epsilon\}|.$$
%By Lemma \ref{l:indcons}, we have the relation $$2+2\lambda_k^\epsilon=3(2+2\lambda_{k-1}^\epsilon)+2+2\lambda_{k-1}^{1-\epsilon},$$ from which the identity follows. An %explicit formula for $|\lambda_k^\epsilon|$ is again safely left as an exercise.
%\end{proof}

%We can now calculate $\lambda^\ep$.

\begin{Lem}\label{l:indcons2}
Let $\lambda^\ep$ be the number of lines that contain any pair of points in the design $\De^\ep$.  Then $$\lambda^\ep=2^{m-2}(2^{m-1}+(-1)^\ep)-1.$$
\end{Lem}

\begin{proof} Let $\theta_w,\theta_z\in \Omega^\ep$, and recall from \eqref{d:overline} the definition of $\overline{\theta_w,\theta_z}$. 
Then $\theta_v\in\overline{\theta_w,\theta_z}$ if and only if $\theta_0(w+z+v)=\ep$.
As $\De^\ep$ is supersimple, $2+2\lambda^\ep=|\overline{\theta_w,\theta_z}|$, and so
$$2+2\lambda^\ep=|\{v\in V_m^\ep\,|\,\theta_0(w+z+v)=\ep\}|.$$
Since $\Sp_{2m}(2)$ acts $2$-transitively on $\Omega^\ep$, we can assume that $w=0$ and $z=e_1$, or $w=e_1+e_2+e_{m+2}$ and $z=e_2+e_{m+2}$ for $\ep=0$ or $1$ 
respectively. In particular, we can assume that $w+z=e_1$. Now, for $v\in V^\ep_m$, $\theta_0(e_1+v)=\ep$ if and only if $v_{m+1}=0$.  
Hence, by Lemma \ref{l:indcons}, $$\{v\in V_m^\ep\,|\,\theta_0(w+z+v)=\ep\}=(V_{m-1}^\epsilon)^{00}\cup (V_{m-1}^\epsilon)^{10},$$ and so 
$2+2\lambda^\ep=2|V_{m-1}^\ep|=2^{m-1}(2^{m-1}+(-1)^\ep)$.  Rearranging this gives the result.
\end{proof}

\subsection{Construction of \texorpdfstring{$\De^a$}{D+-}}
Define $\B^a$ as in the statement of Theorem A:
$$\B^a:=\{\{v_1,v_2,v_3,v_1+v_2+v_3\} \mid v_i \in V, \sum_{i=1}^3 \theta_0(v_i)=\theta_0\left(\sum_{i=1}^3 v_i\right)\}.$$

\begin{Lem}\label{l:aff}
$\De^a:=(V,\B^a)$ is a supersimple $2-(2^{2m},4,2^{2m-2}-1)$ design.
\end{Lem}

\begin{proof}
$\B^a$ contains no repeated lines (by definition) and for each $\ell\in \B^a$, any three points in $\ell$ uniquely determine the fourth, so the intersection of any two lines has size at most 2. Now since $$\sum_{i=1}^3 \theta_0(v_i)=\theta_0\left(\sum_{i=1}^3 v_i\right) \Longleftrightarrow \varphi(v_1,v_2)+\varphi(v_1,v_3)+\varphi(v_2,v_3)=0,$$
$\O_0$ is a $2^{2m}.\Sp_{2m}(2)$-orbit. Also for each $(v,g) \in 2^{2m}.\Sp_{2m}(2)$,

$$(v_1+v_2+v_3)^{(v,g)}=(v_1+v_2+v_3)^g+v^g=\sum_{i=1}^3 (v_i^g+v^g)=\sum_{i=1}^3 v_i^{(v,g)},$$
and we deduce that $\B^a$ is a $2^{2m}.\Sp_{2m}(2)$-orbit on the 4-subsets of $V$.  Since $2^{2m}.\Sp_{2m}(2)$ acts 2-transitively on $V$, $\De^a:=(V,\B^a)$ is a $2-(|V|,4,\lambda)$ design for some $\lambda > 0$ by \cite[Lemma 4.3]{Puzz}, and it remains to calculate $\lambda.$ Now, for each $x,y \in V$, by definition, $$2\lambda+2=|\{z \in V \mid \varphi(x,y)+\varphi(x,z)+\varphi(y,z)=0\}|.$$

Using the 2-transitivity of $2^{2m}.\Sp_{2m}(2)$ again, we may assume that $x=0$ and $y=e_1=(1,0,\ldots,0)$. Hence $2\lambda+2=|\{z \in V \mid \varphi(e_1,z)=0\}|$ which has order $2^{2m-1}$. The result follows.
\end{proof}

\begin{proof}[Proof of Theorem A]
Theorem A follows as an immediate consequence of Lemmas \ref{l:spdes}, \ref{l:indcons}, \ref{l:indcons2} and \ref{l:aff}.
\end{proof}

\section{Infinite familes of Conway groupoids}\label{s:puzz}
In this section, our goal is a description of the Conway groupoids  $\C(\De^\ep)$ and $\C(\De^a)$. 
Taken together, the results of this section yield a proof of Theorem B. First recall the notation $[x,y]$ of Section \ref{sub:design} for a pair $\{x,y\}$ of points in a supersimple $2-(n,4,\lambda)$ design.

\begin{Lem}\label{l:transperm}
The following hold:
\begin{itemize}
\item[(a)] For each $x_0,y_0 \in V^\ep$, the action of $t_{x_0+y_0}$ on $\Omega^\ep$ induces the permutation $[\theta_{x_0},\theta_{y_0}]$.
\item[(b)] For each $x_0,y_0 \in V$, the action of $(y_0^{t_{x_0}}+x_0^{t_{y_0}},t_{x_0+y_0})$ on $V$ induces the permutation $[x_0,y_0]$.
\end{itemize}
\end{Lem}

\begin{proof}
To prove (a), our goal is to show that $t_{x_0+y_0}$ induces the permutation $$\prod_{i=0}^{\lambda^\ep} (\theta_{x_i},\theta_{y_i}),$$ where $\{\theta_{x_0},\theta_{y_0}, \theta_{x_i}, \theta_{y_i}\}$ are the lines in $\B^\ep$ containing $\{\theta_{x_0},\theta_{y_0}\}$ for $1 \leq i \leq \lambda^\ep.$ Note that $\theta_0(x_0+y_0+x_i)=\ep$ and $x_i+y_i=x_0+y_0$, so writing $c:=x_0+y_0$ we have $\theta_{x_i}(c)=\theta_{y_i}(c)=0$. Hence by Lemma \ref{l:trans},
$$\theta_{x_i}^{t_c}=\theta_{x_i+(1+\theta_{x_i}(x_i+y_i))\cdot x_i+y_i}=\theta_{y_i},$$ and similarly for $y_i$. 

Finally, if $z \in V^\ep$ is such that $\theta_0(c+z)=1-\epsilon$ then $$\theta_z(c)=\theta_0(c)+\varphi(c,z)=\theta_0(c+z)+\theta_0(z)=1-\epsilon+\epsilon=1,$$ so that $\theta_z^{t_c}=\theta_z$. This completes the proof of (a).

We next prove (b). Recall from \eqref{e:action} that we regard elements of $2^{2m}.\Sp_{2m}(2)$ as ordered pairs $(v,g)$, so that the action of $2^{2m}.\Sp_{2m}(2)$ on $V$ is described by $$x^{(v,g)} := x^g+v^g $$ for each $x \in V$. We need to show that $h:=(y_0^{t_{x_0}}+x_0^{t_{y_0}},t_{x_0+y_0})$ induces the permutation $$\prod_{i=0}^\lambda (x_i,y_i)$$ where $\{x_0,y_0,x_i,y_i\}$ are the lines in $\B^a$ containing $\{x_0,y_0\}$ for $1 \leq i \leq \lambda$. Observe that $$y_0^{t_{x_0}}+x_0^{t_{y_0}}=(1+\varphi(x_0,y_0))(x_0+y_0),$$ so that 
\begin{align*} x_0^ h = &\,\,x_0^{t_{x_0+y_0}}+(1+\varphi(x_0,y_0))(x_0+y_0)\\
=&\,\,x_0+\varphi(x_0,x_0+y_0)(x_0+y_0)+(1+\varphi(x_0,y_0))(x_0+y_0)\\
 =&\,\,x_0+x_0+y_0= y_0,
 \end{align*} and similarly for $y_0$. Next, for any $1 \leq i \leq \lambda$, we have that $\varphi(x_i,x_0+y_0)=\varphi(x_0,y_0)$, so that 
 
 \begin{align*}
 x_i^h = &\,\,x_i^{t_{x_0+y_0}}+(1+\varphi(x_0,y_0))(x_0+y_0)\\
    = &\,\,x_i+\varphi(x_i,x_0+y_0)(x_0+y_0)+(1+\varphi(x_0,y_0))(x_0+y_0)\\
    = &\,\,x_i+x_0+y_0=y_i\end{align*}
    Lastly, if $r \notin \overline{x_0,y_0}$ then $\varphi(r,x_0+y_0)=\varphi(x_0,y_0)+1$, and a similar calculation shows that $r^h = r$. This completes the proof.
%$$\overline{x_0,y_0}:=\{x_0,y_0\} \cup \{\{x_i,y_i\} \mid x_0+y_0+x_i+y_i=0 \mbox{ for all } 1 \leq i \leq \lambda\}.$$ 
\end{proof}

\begin{Lem}\label{l:transconj}
Let $u,v,w \in V$ be such that \begin{equation}\label{eq6}
\varphi(u,v)+ \varphi(u,w)+\varphi(v,w)=1.
\end{equation} 
The following hold:
\begin{itemize}
\item[(a)] $t_{u+v}^{t_{v+w}}=t_{u+w}.$
\item[(b)] $x^y=z$, where
$$x:=(v^{t_u}+u^{t_v},t_{u+v}) \hspace{3mm} y:=(w^{t_v}+v^{t_w},t_{v+w}) \hspace{3mm}  z:=(w^{t_u}+u^{t_w},t_{u+w})$$ 
\end{itemize}
\end{Lem}
\begin{proof}
To prove (a), we note, by \cite[p.246]{Permutation}, that $u^{t_v}=u+\varphi(u,v)v$ for all $u,v\in V$, and $x^{-1}t_vx=t_{vx}$ for all $x\in \Sp_{2m}(2)$. 
Using (\ref{eq6}), it is straightforward to show that $(u+v)^{t_{v+w}}=u+w$, from which the result now follows.

It remains to prove that $x^y=z$, as in the statement of the lemma. If $(a,b), (c,d) \in 2^{2m}.\Sp_{2m}(2)$, by the usual multiplication rule for the semi-direct product we have: $$(a,b)^{(c,d)}=((c^{-1})^{d} + a^d + c^{(d^{-1}b)^{-1}}, b^d).$$ We will apply in the case where  $$a=v^{t_u}+u^{t_v},\hspace{5mm} b=t_{u+v}, \hspace{5mm} c=w^{t_v}+v^{t_w},\hspace{5mm} d=t_{v+w}.$$ It thus suffices to prove that $(c^{-1})^d + a^d + c^{(d^{-1}b)^{-1}}=w^{t_u}+u^{t_w}$.

 Note that for each $\{i,j,k\} =\{u,v,w\}$, $$(i+k)^{t_{i+j}}=j+k.$$ Using this we obtain the following two equations:
\begin{align*}
a^d=&\,\,(\varphi(u,v)+1)(u+v)^{t_{v+w}}=(\varphi(u,v)+1)(u+w),\\
c^{(d^{-1}b)^{-1}}=&\,\,c^{bd}=(\varphi(w,v)+1)((v+w)^{t_{u+v}t_{v+w}})=(\varphi(w,v)+1)(v+u)).
\end{align*}
Noting also that $(c^{-1})^d=c^d=c=(\varphi(v,w)+1)(v+w)$, we conclude
$$c+a^d+c^{bd}= (\varphi(w,v)+\varphi(u,v))(u+w)=(\varphi(u,w)+1)(u+w)=w^{t_u}+u^{t_w}.$$ 
This completes the proof.
\end{proof}

%(via an abuse of notation) for each $v \in V^\ep$ the form $\theta_v$ with its corresponding vector $v$. Thus we will write ``$v \in \Omega^\ep$'' rather than $``\theta_v \in \Omega^\ep$'' and so on.  In order to avoid confusion, we recall the notation $$\overline{a,b}=\{x\in\Omega^\ep \mid x+a+b\in\Omega^\ep\}$$ for the set of forms collinear with $\{a,b\}$ in $\De^\ep$. Note that the condition $\infty \notin \overline{a,b}$ in both cases is equivalent 
%$$\varphi(\infty,a)+\varphi(\infty,b)+\varphi(a,b)=1.$$

For the remainder of this section we identify, in both cases, the points of the respective design with vectors in $V$, 
allowing us to amalgamate arguments. We now combine Lemmas \ref{l:transperm} and \ref{l:transconj} to obtain the following corollary:

\begin{Cor}\label{c:transconj}
Let $\infty,a,b$ be a triple of points in $\De^\ep$ or $\De^a$ such that $\infty \notin \overline{a,b}$. Then $$[\infty,a]^{[a,b]}=[\infty,b].$$ Consequently, $[\infty,a,b,\infty]=[a,b]$.
\end{Cor}

\begin{proof}
In both cases, $\infty \notin \overline{a,b}$ implies that $$\varphi(\infty,a)+\varphi(\infty,b)+\varphi(a,b)=1.$$ Hence by Lemmas \ref{l:transperm} and \ref{l:transconj}, we have that $[\infty,a]^{[a,b]}=[\infty,b]$. Now two applications of Lemma \ref{l:transconj} yield $$[\infty,a,b,\infty]=[\infty,a][a,b][b,\infty]=[\infty,a][a,b][\infty,a][\infty,a][b,\infty] $$ $$=[\infty,b][\infty,a][b,\infty]=[a,b],$$ as required.
\end{proof}

%\begin{proof}
%Let $d$ be such that $\theta_0(c+b+d)=\delta_1$ and $\theta_0(a+c+d)=\delta_2$. It suffices to show that $$d \cdot (t_{a+b})^{t_{a+c}}=d+(\delta_1+1+\ep)(c+b),$$ so that both $(t_{a+b})^{t_{a+c}}$ and $t_{b+c}$ induce the same permutation matrix by Lemma \ref{l:transperm}.

%Now $\theta_0(x+y+z)=\ep+\varphi(x,y)+\varphi(x,z)+\varphi(y,z)$ for any $x,y,z \in \{a,b,c,d\}$, and using this it is easy to show that $$\theta_0(a+b+d)=1+\ep+\delta_1+\delta_2.$$ From here the claim is easily checked for each of the four pairs $(\delta_1,\delta_2)$ in turn. For example, if $(\delta_1,\delta_2)=(\ep,\ep)$ then $$d \cdot (t_{a+b})^{t_{a+c}}=(d+a+c) \cdot t_{a+b} \cdot t_{a+c}= b+d+c \cdot t_{a+c}=b+d+c.$$
%The lemma is proved.
%\end{proof}

For the designs $\De^\ep$ (respectively $\De^a$) recall that $\L(\De^\ep)$ (respectively $\L(\De^a)$) denote the set of \textit{all} move sequences. We will apply Corollary \ref{c:transconj} to show that the permutation induced by a move sequence in $\L(\De^\ep)$ or $\L(\De^a)$ can be generated with a move sequence which starts with an element of our choosing:

\begin{Lem}\label{l:infstart}
Let $\L$ be either of the sets $\L(\De^\ep)$ or $\L(\De^a)$. For any $g \in \L$ and any point $\infty$, there exist $l > 0$ and a set of points $\{b_1,\ldots, b_l\}$ such that $g=[\infty,b_1,b_2 \ldots, b_l].$ 
\end{Lem}

\begin{proof}
We prove this by induction on the length $k$ of an expression for an element $g:=[a_1,a_2,\ldots,a_k] \in \L$. If $k=2$ then there are two cases to consider. If $\infty \in \overline{a_1,a_2}$ then $[a_1,a_2]=[\infty, \infty+a_1+a_2]$, otherwise $[a_1,a_2]=[\infty,a_1,a_2,\infty]$ by Corollary \ref{c:transconj}. Now assume that $k > 2$. If $\infty \in \overline{a_1,a_2}$ then by induction there exist $l > 0$ and $b_i \in \Omega$ for $1 \leq i \leq l$ such that $[a_2,\ldots, a_k]=[\infty+a_1+a_2,b_1,\ldots,b_l]$ and hence  $g=[\infty,\infty+a_1+a_2,b_1, \ldots, b_l]$. If $\infty \notin \overline{a_1,a_2}$ then there exist $l > 0$ and $b_i \in \Omega$ for $1 \leq i \leq l$ such that $[a_2,\ldots, a_k]=[\infty,b_1,\ldots,b_l]$ so that $g=[\infty,a_1,a_2,\infty, b_1, \ldots, b_l]$. The result follows.
\end{proof}

\begin{Cor}\label{c:lgroup}
Let $\L$ be either of the sets $\L(\De^\ep)$ or $\L(\De^a)$. Then $\L$ is a group. Furthermore, for each point $\infty$, we have $\L=\L_\infty(\De)$.
\end{Cor}

\begin{proof}
$\L$ clearly contains the trivial move sequence and $[a_1,a_2,\ldots,a_r]^{-1}=[a_r,a_{r-1},\ldots,a_1]$ for each $[a_1,a_2,\ldots,a_r] \in \L$. It remains to show that $\L$ is closed under composition.

For $1 \leq i \leq k$ and $1 \leq j \leq l$ let $a_i,b_j$ be points and write $g:=[a_1,\ldots,a_k]$ and $h:=[b_1,\ldots,b_l].$ By Lemma \ref{l:infstart}, there exist $s > 0$ and points $c_1, \ldots c_s$ such that $[b_1,\ldots,b_l]=[a_k,c_1,\ldots, c_s]$ so we have $g \cdot h=[a_1,\ldots,a_k,c_1,\ldots,c_s] \in \L,$ as required. 

The last statement follows immediately from Lemma \ref{l:infstart}.
\end{proof}

%We can now complete the proof of Theorem B. The proof of one of the results we require (Lemma \ref{l:sumof2}) is deferrd until Section \ref{s:codes}:

%\textbf{Note below that $\pi_0(\D)= \Sp_{2m}(2)$ in the affine case.}

\begin{proof}[Proof of Theorem B]
First note that, by Corollary \ref{c:lgroup}, $\L(\De^\ep)$ and $\L(\De^a)$ are both groups. In both cases, this group is generated
by all elementary move sequences $[a,b]$. In this first case, $[a,b]=t_{a+b}$ for all $a,b\in\Omega^\ep$ (Lemma \ref{l:transperm}), and since
every $v\in V$ can be written as the sum of two elements in $V^\ep$ (Lemma \ref{l:sumof2}), it follows that
$$\L(\De^\ep)=\langle t_{a+b}\,|\,a,b\in V^\ep\rangle=\langle t_{v}\,|\,v\in V\rangle\cong \Sp_{2m}(2).$$

In the second case, $[a,b]=(a^{t_b}+b^{t_a}, t_{a+b})$, by Lemma \ref{l:transperm}, and so $\L(\De^a)\leq 2^{2m}.\Sp_{2m}(2)$.  Now, it is straightforward to show that for 
every $0\neq a\in V$, there exists $x_a\in V$ such that $\varphi(x_a,a)=1$. One then calculates that $[x_a,x_a+a]=(0,t_a)$ for all $0\neq a\in V$, 
so $\Sp_{2m}(2)\cong\langle [x_v,x_v+v]\,|\, v\in V\rangle\leq \L(\De^a)$. To get the translations of the affine group, we observe that $[0,a]=(a,t_a)$ for
all $a\in V$, so $[0,a][x_a,x_a+a]=(a,1)$. Thus
$$\L(\De^a)\cong 2^{2m}.\Sp_{2m}(2). $$
Now let $\L$ denote $\L(\De^\ep)$ (respectively $\L(\De^a))$ and $\pi$ denote $\pi_\infty(\De^\ep)$ (respectively $\pi_\infty(\De^a)).$ By \cite[Lemma 3.1]{Puzz}, $|\L|=n \cdot |\pi|$ where $n$ is the number of points in the associated design, and since $\pi \subseteq \stab_\L(\infty)$ we must have an equality $\pi=\stab_{\L}(\infty)$. This completes the proof.
 %\cong \left\{
%	\begin{array}{ll}
%		\Oo_{2m}^+(2),  & \mbox{if } \ep = 0; \\
%		\Oo_{2m}^-(2), & \mbox{if } \ep = 1.
%	\end{array}
%\right.$$ 
%Since $H$ is generated by transvections $t_c$ (Lemma \ref{l:trans} (iii)) and each $c \in V$ may be written as the sum of two elements of $V^\epsilon$, we must have $H=\langle [\theta_x,\theta_y] \mid x,y \in V^\epsilon \rangle$ by Lemma \ref{l:transperm}. Hence $\L_{\De^\ep} \subseteq H$.
\end{proof}

%In an arbitrary  $2-(n,k,\lambda)$ design $\De$, recall that a set of $s$ points in $\De$ is an \textit{arc} if no three of them are collinear. An \textit{oval} is an arc of maximal size.  Let $\O=\O(\De)$ denote the collection of ovals.

%\begin{Lem}\label{l:ovals}
%Let $\De$ be a $2-(n,k,\lambda)$ design and let $C:=C_{\mathbb{F}_p}(\De)$ be its code. If $C \neq \mathbb{F}_p^n$ then $$d(C^\perp) \geq \frac{r+\lambda}{\lambda}$$  
%\end{Lem}

%\begin{proof}
%See \cite[Lemma 2.4.2]{AsK}.
%\end{proof}

%\begin{Lem}
%$r$ is odd.
%\end{Lem}

%\begin{proof}
%$$r=\frac{(2^{2m-1} + (-1)^{\epsilon 1} 2^{m-1}-1)\cdot (2^{2m-3}+ (-1)^{\epsilon 1} 2^{m-2}-1)}{3},$$ so that it is enough to show that either $2^{2m-1}+ (-1)^{\epsilon 1} 2^{m-1}-1$ or $2^{2m-3}+ (-1)^{\epsilon 1} 2^{m-2}-1$ is divisible by $9$, which is trivial to check.
%\end{proof}

%\begin{Cor}
%The all--one vector $\textbf{1} \in C^\epsilon$.
%\end{Cor}

%\begin{proof}
%Since $r$ is odd, $$\sum_{B \in \B^\epsilon} v_B =  \textbf{1}.$$
%\end{proof}

\section{Infinite Families of Completely Transitive Codes}\label{s:codes}
This section is concerned with the $\mathbb{F}_2$-linear codes $C^\ep:=C_{\mathbb{F}_2}(\De^\ep)$ and $C^a:=C_{\mathbb{F}_2}(\De^a)$ associated respectively to the incidence matrices of the designs $\De^\ep$ and $\De^a$ of Theorem A. (Recall that $C_{\mathbb{F}_2}(\mathcal{E})$
 is simply the $\mathbb{F}_2$-rowspan of the the incidence matrix of a design $\mathcal{E}$.)
We first introduce some notation which will allow us to describe elements of $C^\ep$ and $C^a$ succinctly.
%For $\epsilon \in \mathbb{F}_2$, let $W^\epsilon$ be the $|V^\epsilon|$-dimensional vector space over $\mathbb{F}_2$ with basis $\{w_p \mid p \in V^\epsilon\}$. For a subset $\S$ of $V^\epsilon$, let $w_\S$ denote the vector sum $$w_\S:=\sum_{s \in \S} w_s \in W^\epsilon.$$ 
For $\epsilon \in \mathbb{F}_2$, let $W^\epsilon$ be the $|\Omega^\epsilon|$-dimensional vector space over $\mathbb{F}_2$ with entries indexed by
$\Omega^\ep$ and $W^a$ be the $|\Omega|$-dimensional vector space over $\mathbb{F}_2$ with entries indexed by
$V$.  Therefore, each element $\alpha_{\S}$  of $W^\ep$ or $W^a$ can be uniquely identified with a subset $\S$ of $\Omega^\ep$ or $V$, that is, $\alpha_\S$ is
the characteristic vector of $\S$. Thus, we note that $\supp(\alpha_\S)=\S$ ($\supp(v)$ denotes the \emph{support} of a vector $v$, that is, the set of
non-zero entries of $v$). Using this notation
$$C^\ep=\langle \alpha_\S \mid \S\in\B^\ep \rangle \mbox{ and } C^a:=\langle \alpha_\T \mid \T\in\B^a \rangle$$
In particular, for $\S \in \B^\ep$ and $\T \in \B^a$,
\begin{equation}\label{sum0} \sum_{\theta_a\in\S}a=0,\end{equation}
and 
\begin{equation}\label{sum01}\sum_{a\in\T}a=0,\,\,\,\,\sum_{a\in\T}\theta_0(a)=0 \end{equation}

\begin{Lem}\label{l:sum} 
The respective expression \eqref{sum0}, \eqref{sum01} holds for all $\alpha_\S \in C^\ep$, $\alpha_\T \in C^a$, where $\S \subseteq \Omega^\ep$ and $\T \subseteq V$.
\end{Lem}

\begin{proof} 
Let $\alpha_\X$ and $\alpha_\Y$ be two vertices in $W^\ep$ with $$\sum_{\theta_a\in\X}a = \sum_{\theta_a\in\Y}a=0.$$ As $\supp(\alpha_\X+\alpha_\Y)=\X\vartriangle\Y$, the symmetric difference of $\X$ and $\Y$, it follows that
$$\sum_{\theta_a\in\X\vartriangle\Y}a=\sum_{\theta_a\in\X\vartriangle\Y}a+\sum_{\theta_a\in\X\cap\Y}2a=\sum_{\theta_a\in\X}a+\sum_{\theta_a\in\Y}a=0.$$
Since \eqref{sum0} holds for all $\alpha_\S$ such that $\S\in\B^\ep$, the assertion now follows. An analogous argument shows that \eqref{sum01} holds
for all $\alpha_{\T}\in C^a$.
\end{proof}

\begin{Cor}\label{c:dist} $C^\ep$ and $C^a$ consist entirely of codewords of even weight and both codes have minimum distance $d=4$. Moreover, the sets of codewords of weight $4$ are in bijection with $\B^\ep$ and $\B^a$ respectively and 
$$C^\ep=\langle \alpha_\S \mid |\S|=4, \sum_{\theta_a\in\S} a = 0 \rangle, \mbox{ and } C^a=\langle \alpha_\T \mid |\T|=4, \sum_{a\in\T} a = 0, \sum_{a \in \T} \theta_0(a)= 0 \rangle.$$
\end{Cor}

\begin{proof} We only treat the code $C^a$, since a similar argument also holds for $C^\ep$. As $C^a$ is generated by codewords with weight $4$,
it follows that it consists entirely of codewords with even weight.  
Suppose there exists $\alpha_\T\in C^\ep$ with weight $2$, so $\T=\{v,w\}$ for some $v\neq w$.  Then
Lemma \ref{l:sum} implies that $v+w=0$, a contradiction, hence $d=4$. Now let $\alpha_\T$ be any weight $4$ vertex in
$W^a$ that satisfies \eqref{sum01}, with $\T=\{v_1,v_2,v_3,v_4\}$.
Then \eqref{sum01} implies that $v_4=v_1+v_2+v_3$ and $\sum_{i=1}^3\theta_0(v_i)=\theta_0(\sum_{i=1}^3v_i)$. In particular, 
$\T\in\mathcal{B}^a$. Now, by Lemma \ref{l:sum}, 
all codewords of weight $4$ satisfy \eqref{sum01}, which proves the second statement.
\end{proof}

%Hence $$C^\ep=\langle w_\S \mid |\S|=4, \sum_{s \in \S} s = 0 \rangle.$$

%\begin{Lem}\label{l:spdes2}
%$C^\ep$ has minimum distance 4 and the set of codewords of weight 4 is in bijection with $\B^\ep.$ 
%\end{Lem}

%\begin{proof}
%Both assertions follow immediately from the fact that each codeword in $C^\ep$ corresponds to a set of quadratic forms with zero sum.
%\end{proof}

%Our next result asserts that every set of evenly many forms with zero sum arises as a codeword.

\subsection{Covering radius and complete transitivity}
We next give succinct descriptions of the codewords of $C^\ep$ and $C^a$.

\begin{Lem}\label{l:codechar}
The following hold:
\begin{itemize}
\item[(a)] For $m \geq 4$ and $\alpha_\S\in W^\ep$, $\alpha_\S\in C^\ep$
if and only if $|\S|=2k$ for some $k\geq 2$ and $\sum_{\theta_a \in \S} a = 0$.
\item[(b)] Let $m \geq 3$ and $\alpha_\S \in W^a$. Then $\alpha_\S \in C^a$ if and only if $|\S|=2k$ for some $k \geq 2$, $\sum_{s \in \S} s = 0$ and $\sum_{s \in \S} \theta_0(s)= 0$.
\end{itemize}
\end{Lem}

\begin{proof}
%The case $m=3$ is handled by a GAP computation, so we assume throughout the proof that $m \geq 4$. 
In both (a) and (b), the forward implication is a consquence of Lemma \ref{l:sum} and Corollary \ref{c:dist}, and the reverse
implication for $k=2$ also follows from Corollary \ref{c:dist}. We first prove the reverse implication for (a).
Suppose we have verified the claim when $k=3$ and assume (by induction) that the claim holds for all $\S$ with $|\S|=2\ell$ and $\ell < k$. Write $\alpha:=\alpha_\S$ for short and assume that $k > 3$. If there exist $\theta_x,\theta_y,\theta_z \in \S$ such that $\theta_0(x+y+z)= \epsilon$ then $\alpha_{\S'}\in C^\ep$ where 
$\S'=\{\theta_x,\theta_y,\theta_z,\theta_{x+y+z}\}$. Since $|\supp(\alpha+\alpha_{\S'})|  <2k$,
it follows from Lemma \ref{l:sum} that $\alpha+\alpha_{\S'}$ statisfies the inductive hypothesis.  Thus $\alpha+\alpha_\S'\in C^\ep$, and so $\alpha\in C^\ep$.

We are thus reduced to the case where  
\begin{equation}\label{transprop}
\theta_0(x+y+z)= 1-\epsilon \mbox{ for all } \theta_x,\theta_y,\theta_z \in \S.
\end{equation}

Now, for any $\theta_x,\theta_y,\theta_z,\theta_s \in \S$, there exist $t,u \in V^\epsilon$ such that $x+y+z+s=t+u$ by Lemma \ref{l:sumof2}.
Furthermore, by \eqref{transprop}, we must have $\{x,y,z,s\} \cap \{t,u\} = \emptyset.$ By induction,
$\alpha_{\S'}\in C^\ep$ where $\S'=\{\theta_x,\theta_y,\theta_z,\theta_s,\theta_t,\theta_u\}$.
Moreover, $|\supp(\alpha+\alpha_{\S'})| < |\S|$, therefore, as before, $\alpha\in C^\ep$.

It thus remains to verify the claim in the case where $k=3$. Since $6 > 4=2^2$ at least 3 of the vectors associated with the forms in $\S$ are linearly independent. Since the sum of 6 distinct vectors in $\mathbb{F}_2^3$ cannot be $0$, at least 4 of the vectors associated with the forms in $\S$ are linearly independent. Further, an identical argument to that given in the first paragraph shows that we may assume  (\ref{transprop}) holds for $\S$.

Let $\{a_1,a_2,a_3,a_4\}$ be the four linearly independent vectors, so that $\S=\{\theta_{a_1},\theta_{a_2},\theta_{a_3},\theta_{a_4},\theta_r,\theta_s\}$ for some $\theta_r,\theta_s\in\Omega^\ep$. By the pigeonhole principle there exist two equal elements in the set $\{\varphi(a_1,a_3),\varphi(a_2,a_3),\varphi(a_4,a_3)\}$, $\varphi(a_1,a_3)$ and $\varphi(a_2,a_3)$ say. By Lemma \ref{l:linind} we may choose $$x \in L(a_1+a_2,\theta_0(a_1+a_2)) \cap L(a_3,\varphi(a_1,a_3)+1) \cap L(a_3+a_4,\theta_0(a_3+a_4)),$$ so that $\theta_0(x)=\ep$. This implies that $x \notin \{a_1,a_2\}$ and since $$\theta_0(x+a_1+a_2)=\theta_0(x)+\theta_0(a_1+a_2)+\varphi(x,a_1+a_2)=\theta_0(x)=\ep,$$ $\S'=\{\theta_x,\theta_{x+a_1+a_2},\theta_{a_1},\theta_{a_2}\}$ is the support of some codeword $\alpha_{\S'}$. Now, as \eqref{transprop} holds, $0$ or $1$ elements in the set $\{x,x+a_1+a_2\}$ lie in $\{a_3,a_4,r,s\}$. If it is $1$ then we must have $\alpha+\alpha_{\S'} \in C^\ep$, so that $\alpha \in C^\ep$. If it is $0$ then $$\supp(\alpha+\alpha_{\S'})=\{\theta_x,\theta_{x+a_1+a_2},\theta_{a_3},\theta_{a_4},\theta_r,\theta_s\},$$ and \begin{align*}\theta_0((x+a_1+a_2)+r+s)&=\theta_0(x+a_3+a_4)\\
&=\theta_0(x)+\theta_0(a_3+a_4)+\varphi(x,a_3+a_4)=\theta_0(x)=\ep\end{align*} so that $\alpha+\alpha_{\S'}=\alpha_\T+\alpha_\U$ where $$\T:=\{\theta_{x+a_1+a_2},\theta_r,\theta_s,\theta_{x+a_1+a_2+r+s}\} \mbox{ and } \U:=\{\theta_x,\theta_{a_3},\theta_{a_4},\theta_{x+a_3+a_4}\}.$$ Clearly both  $\alpha_\T$ and $\alpha_\U$ lie in $C^\ep$, so that  $\alpha \in C^\ep$ in this case also. This completes the proof of (a).

We next prove the reverse implication for (b). Note first that $$\sum_{s \in \S} \theta_0(s)=\sum_{s \in \S'} \theta_0(s)=0 \Rightarrow \sum_{s \in \S \triangle \S'} \theta_0(s)=0.$$ The proof is again by induction on $k$, where we assume that $k \geq 3$. Since $\sum_{s \in \S} s = 0$, $\langle s \mid s \in \S \rangle$ is a subspace of dimension at least 4 and we may pick be 4 linearly independent vectors in $\S$,  $a_1,a_2,a_3,a_4$ say. By the pigeonhole principle, there exist two elements $s,t  \in \{a_1,a_2,a_3,a_4\}$, with $\theta_0(s)=\theta_0(t)$. Without loss of generality, $s=a_1$, $t=a_2$. By \cite[Lemma 7.7B]{Permutation}, we may choose $$x \in L(a_1+a_2,\varphi(a_1,a_2)) \cap L(a_3+a_4,\varphi(a_3,a_4))$$ with the property that $\theta_0(x) \neq \theta_0(a_1)$. Note in particular that this implies that $x \notin \{a_1,a_2\}$ and $$\theta_0(x+a_1+a_2)=\theta_0(x)+\theta_0(a_1)+\theta_0(a_2).$$

If $|\S \triangle \{x,a_1,a_2,x+a_1+a_2\}| < 2k$, then the lemma holds by induction. Otherwise we have $x \notin \{a_3,a_4\}$ and 
$$\theta_0(x+a_3+a_4)=\theta_0(x)+\theta_0(a_3)+\theta_0(a_4),$$
and hence  $|(\S \triangle \{x,a_1,a_2,x+a_1+a_2\}) \triangle \{x,a_3,a_4,x+a_3+a_4\}| < 2k$ and the lemma holds by induction in this case also. The proof is complete.

%In fact, since $\theta_0(\sum_{i=1}^5 a_i\})=\epsilon,$ $\theta_0(a_i+a_j+a_k)=\theta_0(a_l+a_m)$ where $\{a_1,a_2,a_3,a_4,a_5\}=\{a_i,a_j,a_k,a_l,a_m\}$, so our assumption is equivalent to: $$\varphi(a_i,a_j)=\theta_0(a_i+a_j) = 1-\epsilon \mbox{ for all } 1 \leq i < j \leq 5.$$
%By Lemma \ref{l:linind}, we may choose $x \in L(a_1,1-\epsilon) \cap L(a_2,\epsilon) \cap L(a_3,\epsilon)$ such that $\theta_0(x)=\epsilon$. Then $$\theta_0(x+a_1+a_2)=\varphi(x,a_1)+\varphi(x,a_2) + \varphi(a_1,a_2)=\epsilon.$$ (Observe in particular that these conditions imply $x \notin \{a_1,a_2\}$.)  Thus $\{x,a_1,a_2,x+a_1+a_2\}$ is the support of some codeword $c_1$ of weight $4$.  Now,  $x, x+a_1+a_2 \notin \{a_4,a_5\}$ since $\varphi(x,a_2)=\epsilon$ and $\varphi(x+a_1+a_2,a_3)=\epsilon$ and $c+c_1$ is a codeword of weight $4$ or $6$. In the first case we are done. Otherwise, $$\supp(c+c_1)=\{x,a_3,a_4,a_5,\sum_{i=1}^5 a_i,x+a_1+a_2\},$$
%and $$\epsilon=\varphi(x,a_3)=\theta_0(x+a_3)=\theta_0(a_4+a_5+\sum_{i=1}^5 a_i)=\theta_0(a_1+a_2+a_3),$$ a contradiction. This completes the proof.
\end{proof}

We next show how to identify a certain code constructed in \cite{BRZ} with our code $C^a$. This will allow us to prove Theorem C (b). We begin by reviewing the construction from \cite{BRZ}. Let $\textbf{H}$ be the $2m \times (2^{2m}-1)$ parity check matrix of the $[2^{2m}-1,2^{2m}-2m-1,3]$ linear Hamming code, whose
whose columns $c_1,\ldots, c_{2^{2m}-1}$ correspond to the non-zero elements of $(\mathbb{F}_2)^{2m}$. Let $q:V \rightarrow \mathbb{F}_2$ be the ``bent function'' defined by \[
 q(v)=\left\{\begin{array}{ll}
           1, & \textrm{ if $wt(v) \equiv 2,3 \mod 4$;} \\
           0, & \textrm{ otherwise.}
          \end{array}\right.
\]

In \cite{BRZ} it is shown that $q$ is quadratic, that is, $q(v+w)+q(v)+q(w)$ is a bilinear form on $V$, and that $q(v)=vQv^T$ where $Q$ is the all ones upper triangular matrix.

Let $x$ be the row vector of length $2^{2m}-1$ with $x_i:=q(c_i)$ and form a new matrix $\textbf{H}_x$ from $\textbf{H}$ by letting $x$ be an additional row. 
Now let $C_x$ be the code that has $\textbf{H}_x$ as its parity check matrix, and let $C$ be the \emph{extended code} of $C_x$, that is, the code obtained from 
$C_x$ via the addition of an extra coordinate so that the sum of the coordinates of the extended codeword is zero.  Thus codewords in $C$ may be identified with vectors 
$\alpha_\S$ ($\S \subseteq V$) with the property that $$|\S|\textnormal{ is even, }\sum_{v \in \S} v=0 \mbox{ and } \sum_{v \in \S} q(v)=0.$$

Now, there exists an invertible $2m \times 2m$ matrix $A$ such that $AeA^T=Q$ where $e$ is the matrix of Section \ref{s:action} used to define the form $\theta_0$. We claim that the map $\rho_A: C \rightarrow C^a$ which sends a codeword $\alpha_\S \in C'$ (some $\S \subseteq V$) to $\alpha_\T$ where $\T:=\{sA \mid s \in \S\}$ defines an isomorphism of codes. Indeed, if $\alpha_\S \in C$ then $|\S|=|\T|$ is even, 
$$\sum_{u \in \T} u = \sum_{v \in \S} vA= (\sum_{v \in \S} v)A=0, \mbox{ and }$$

$$\sum_{u \in \T} \theta_0(u) = \sum_{v \in \S} vAeA^Tv^T = \sum_{v \in \S} q(v)=0.$$
%Since $f(v)=\theta_a(v)=\theta_0(v)+\varphi(a,v)$, we have $$0=\sum_{v \in \S} f(v)=\sum_{v \in \S} \theta_0(v)+\varphi(a,v)=\sum_{v \in \S} \theta_0(v)+\varphi(\sum_{v \in \S} v,a)=\sum_{v \in \S} \theta_0(v).$$ 

Hence, by Lemma \ref{l:codechar} (b), $C^a$ is isomorphic to $C$, and using results in \cite{BRZ} we obtain:

\begin{Thm}\label{t:affcode}
$C_{\mathbb{F}_2}(\De^a)$ is a completely transitive $[2^{2m},2^{2m}-(2m+2),4]$ code with covering radius $4$ and intersection array $$(2^{2m},2^{2m}-1,2^{2m-1},1;1,2^{2m-1},2^{2m}-1,2^{2m}).$$
\end{Thm}

\begin{proof}
This follows from \cite[Theorem 2.4]{BRZ}.
\end{proof}

We are left with the task of proving Theorem C (a). Recall from Section \ref{s:back} the notation $$C^\ep_i:=\{\beta \in W^\epsilon \mid \min_{\alpha \in C^\ep} d(\beta,\alpha) = i\}.$$ Our next result shows that $C^\ep_i = \emptyset$ for all $i \geq 4$ (so $C^\ep_i$ has covering radius 3) from which we can quickly deduce that $C^\ep$ is a completely transitive code.

\begin{Prop}\label{p:covC}
Let $m \geq 4$ and $\epsilon \in \mathbb{F}_2$. For each $\alpha_\S \in W^\ep$ with $\S:=\supp(\alpha_\S)$ and $v:=\sum_{a \in \S} a$, one of the following holds: 
\begin{itemize}
\item[(i)] $|\S|$ is even, $v = 0$ and $\alpha_\S \in C^\ep_0$;
\item[(ii)] $|\S|$ is odd, $v \in V^\ep$ and  $\alpha_\S\in C^\ep_1$;
\item[(iii)] $|\S|$ is even, $v \neq 0$ and $\alpha_\S \in C^\ep_2$;
\item[(iv)] $|\S|$ is odd, $v \in V^{1-\ep}$ and  $\alpha_\S \in C^\ep_3$.
\end{itemize}
Consequently, $C^\ep$ has covering radius 3.
\end{Prop}

\begin{proof}
Suppose that $|\S|$ is even. If $v=0$, then by Lemma \ref{l:codechar} $\alpha_\S\in C^\ep$ and (i) holds, so we may assume that $v\neq 0$. 
By Lemma \ref{l:sumof2}, $v=x+y$ for distinct elements $x,y\in V^\ep.$
Set $\alpha':=\alpha_\S+\alpha_{\S'}$, where $\S'=\{\theta_x,\theta_y\}$,  so that $$\supp(\alpha')=\S\Delta\S' \mbox{ and } \sum_{\theta_a\in\supp(\alpha')} a=0.$$ In particular, $\alpha'\in C^\ep$ and $d(\alpha_\S, \alpha')=2$.  Now (iii) follows because $C^\ep$ has minimum distance $d=4$.  

Next suppose that $|\S|$ is odd.  If $v \in V^\ep$ then $\alpha'=\alpha_\S+\alpha_{\{\theta_v\}}$ is a codeword with $d(\alpha_\S,\alpha')=1$ so that (ii) holds. If $v \in V^{1-\ep}$ then  by Corollary \ref{c:sumof2}, there exist $x,y,z \in V^\ep$ such that $v=x+y+z$. In this case $\alpha'=\alpha_\S+\alpha_{\{\theta_x,\theta_y,\theta_z\}}$ is a codeword with $d(\alpha_\S,\alpha')=3$ and (iv) holds.

%It remains to prove that $C^\ep_3 \neq \emptyset$. To see this, let $v$ be any non-zero element of $V^{1-\epsilon}$. By Corollary \ref{c:sumof2}, there exist
%pairwise distinct elements $v_1,v_2,v_3\in V^{\ep}$ such that $v=v_1+v_2+v_3$. Let $w\in W^{\ep}$ with $\supp(w)=\{v_1,v_2,v_3\}$.  Since $wt(w)$ is odd, and all codewords have even weight, $w \in C^\ep_i$ for some odd integer $i > 0$. If $w\in C^\ep_1$ then $d(c,w)=1$ for some $c\in C^\ep$ so that $wt(c)=4$ and $\supp(w)\subset \supp(c)$. This would imply that $v_1+v_2+v_3\in V^{\ep}$, a contradiction. Thus since $\rho(C^\epsilon) \leq 3$, $w\in C^\ep_3$, as needed.

%let $w \in W^\ep$ be any vertex with $wt(w)=3$ such that $\sum_{a \in \supp(w)} a \in V^{1-\epsilon}$. (It is elementary to check that such a vertex exists.) 
%Let $v$ be any non-zero element of $V^{1-\epsilon}$. It follows from Lemma \ref{l:sumof2} that $v=v_1+v_2$ for some $v_1,v_2\in V^{\epsilon}$. Moreover,
%Hence $w\in C^\ep_i$ for some $i\geq 2$.
%However, because $w$ has odd weight, and all codewords have even weight, $i$ must be odd. Thus, because $C^\ep$ has covering radius at most $3$,
%we conclude that  and so $C^\ep$ has covering radius $3$. 
\end{proof}

\begin{Cor}\label{c:comtrans}
For each $m \geq 3$ and $\epsilon \in \mathbb{F}_2$, $C^\ep$ is a completely transitive code with covering radius $3$.  
\end{Cor}

\begin{proof}
By Proposition \ref{p:covC}, $C^\ep$ has covering radius $3$ for $m\geq 4$, and using GAP \cite{GAP}, we verify this to
hold when $m=3$ also. Thus we need to show that $\Aut(C^\ep)$ is transitive on $C_i^\ep$ for $i=0,1,2,3$.  
Since $C^\ep$ is generated by the rows of the incidence matrix of $\De^\ep$, and
because $\De^\ep$ is a $\Sp_{2m}(2)$-orbit, it follows that $\Aut(C^{\epsilon}) \geq N_{C^{\epsilon}}\rtimes \Sp_{2m}(2)$, where $N_{C^{\epsilon}}$ is the group of translations of $C^{\epsilon}$. As $N_{C^{\epsilon}}$ acts regularly on $C^\ep$, $\Sp_{2m}(2)$ acts $2$-transitively on entries and $C^\ep$ has minimum distance $d=4$, 
we deduce that $C^\ep$, $C^\ep_1$ and $C^\ep_2$ are all $\Aut(C^\ep)$-orbits.
Let $\nu_1,\nu_2\in C^\ep_3$. As $\Aut(C^\ep)$ acts transitively on $C^\ep$, we can assume that 
$\nu_1, \nu_2 \in \Gamma_3(0)\cap C^\ep_3$. (Recall that $\Gamma_i(\alpha)=\{\beta\in W^\ep\,|\,d(\beta,\alpha)=i\}$.)
It is straightforward to show that both $\Gamma_3(0)\cap C_1$ and $\Gamma_3(0)\cap C_3$ are non-empty sets.
Thus $\Sp_{2m}(2)$ has at least $2$ orbits on $\Gamma_3(0)$.  But, by Theorem \ref{t:3sets}, $\Sp_{2m}(2)$ has exactly two orbits on $\Gamma_3(0)$.
Hence there exists $g\in \Sp_{2m}(2)$ such that $\nu_1^{g}=\nu_2$, proving that $C^\ep_3$ is an $\Aut(C^\ep)$-orbit, and therefore, $C^\epsilon$ is completely transitive.
\end{proof}

%\begin{Lem}\label{l:autc}
%For each $\epsilon \in \mathbb{F}_2$, $$\Aut(C^\epsilon) = N_{C^{\epsilon}}\rtimes \PermAut(C)\geq N_{C^{\epsilon}}\rtimes G.$$
%\end{Lem}

%\begin{proof}
%This is an immediate consequence of 
%\end{proof}

%We now prove that $C^\epsilon$ is a completely transitive code.

%\begin{Cor}\label{c:des}
%The supports of weight 4 vectors in $C^\epsilon$ form the block set for a $2-(n^\epsilon,4,\lambda_\epsilon)$ design $\De^\epsilon$.
%\end{Cor}

%\begin{proof} Let $n=n^\ep$.  As $C^\ep$ is completely transitive (and therefore completely regular) with minimum distance $4$, 
%a result of van Tilborg \cite[Thm. 2.4.7]{vantil} implies that $C(4)$ forms a $2-(n,4,\lambda_\ep)$ for some positive integer $\lambda_\ep$.  
%Now, each codeword of weight $4$ covers exactly $4$ elements of $\mathcal{O}^\ep_\ep$, and each element of $\mathcal{O}^\ep_\ep$ is
%adjacent to a unique codeword of weight $4$.  Thus $$|\mathcal{O}^\ep_\ep|=4|C(4)|=\frac{n(n-1)\lambda}{3},$$ so, by Lemma \ref{l:sizeoforbs}, 
%$\lambda=\lambda_\ep=2^{m-2}(2^{m-1}+(-1)^\ep)-1$.  
%\end{proof}

\subsection{Dimension of \texorpdfstring{$C^\ep$}{C epsilon} and completing the proof of Theorem C }

%\begin{Lem}
%The set of codewords of weight 4 in $C^\epsilon$ forms a $2-(n,4,\lambda)$ design where $\lambda=2^{m-2} \cdot (2^{m-1} +(-1)^\epsilon)-1$.
%\end{Lem}
%\begin{proof}

%textit{(b)}: \textbf{PROVE THIS PROPERLY} It remains to show that $\lambda=f(m-1)-1$. Since $\Sp_{2m}(2)$ acts 2-transitively on $\Omega^+$, it suffices to show that the pair $0,\underline{x}$, where $\underline{x}:=(1,0,0,...)$ is contained in the stated number of lines. This amounts to showing that there are $2 \cdot f(m-1)-2$ vectors $\underline{v}$ (excluding $0$ and $\underline{x}$) with the property that $\underline{x}+\underline{v} \in V^0$. Equivalently, that there are $2 \cdot f(m-1)$ such vectors $\underline{v}$ with a $0$ in the $(m+1)^{st}$ coordinate. This may be demonstrated via an induction, based on the construction of $V^0$ given in the previous paragraph. A similar argument works for $V^0$.
%\end{proof}

By Proposition \ref{p:covC}, we must have \begin{equation} \label{eq1}
2^{n^\ep}=|W^\ep|=|C^\ep|\sum_{i=0}^3\mu_i.\end{equation} where $n^\ep=|\Omega^\ep|$ and $\mu_i$ denotes
the number of cosets of $C^\ep$ of weight $i$. Thus, in the next result, we calculate
$\mu_i$ for $i=0,1,2,3$ which allows us to determine the dimension of $C^\ep$.

\begin{Prop}\label{p:dimC}
For each $m \geq 3$ and $\epsilon \in \mathbb{F}_2$, let $f_\epsilon(m):=2^{m-1} \cdot (2^m + (-1)^\epsilon).$ Then 
$C^\ep$ is a $[f_\epsilon(m),f_\epsilon(m)-(2m+1),4]$ completely transitive code with 
intersection array $$(f_\epsilon(m),f_\epsilon(m)-1,f_\epsilon(m)-2f_\ep(m-1);1,2f_\ep(m-1),f_\ep(m)).$$
\end{Prop}

\begin{proof}
Write $n^\ep:=f_\epsilon(m)$ for short. By Proposition \ref{p:covC} and Corollary \ref{c:comtrans}, 
$C^\ep$ is completely transitive (and therefore completely regular) with covering radius $3$.
Let $(b_0,b_1,b_2; c_1,c_2,c_3)$ be the intersection array of $C^\ep$.  As $C^\ep$ has minimum distance $d=4$, 
it follows that $b_0=n^\ep$, $b_1=n^\ep-1$ and $c_1=1$. As $C^\ep$ is generated by codewords of weight $4$, it consists entirely of codewords of even weight.
From this we deduce that for any $\nu\in C_i$, there are no neighbours of $\nu$ in $C_i$, that is, $n^\ep-b_i-c_i=0$ 
(so $b_i+c_i=n^\ep$) for $i=0,1,2,3$. Therefore $c_3=n^\ep$. Now let $\nu\in C^\ep_2$,
and without loss of generality, assume that $\nu$ has weight $2$. Clearly $\nu$ has exactly two neighbours of weight $1$ in $C^\ep_1$,
so the number $c_2-2$ is equal to the number of weight $3$ neighbours of $\nu$ that are also covered by a codeword of weight
$4$. By Corollary \ref{c:dist}, the codewords of weight $4$ form a $2-(n^\ep,4,\lambda^\ep)$ design where $\lambda^\ep=2^{m-2}(2^{m-1}+(-1)^\ep)-1=f_\ep(m-1)-1$, 
so there exist $\lambda^\ep$ codewords of weight $4$ that cover $\nu$. Each contributes $2$ neighbours of $\nu$ of weight $3$ 
that are in $C^\ep_1$. Hence $c_2=2\lambda^\ep+2$, and thus 
$b_2=n^\ep-2\lambda^\ep-2$. Applying Lemma \ref{l:rifa} to the intersection array gives 
$$\mu_0=1,\,\,\mu_1=n^\ep,\,\,\mu_2=\frac{n^\ep(n^\ep-1)}{2\lambda^\ep+2},\,\,\mu_3=\frac{(n^\ep-1)(n^\ep-2\lambda^\ep-2)}{2\lambda^\ep+2}.$$
Thus $$2^{n^\ep}=|W^\epsilon|=|C^\ep|(1+n^\ep+\frac{n^\ep(n^\ep-1)}{2\lambda^\ep+2}+\frac{(n^\ep-1)(n^\ep-2\lambda^\ep-2)}{2\lambda^\ep+2}).$$
But $$n^\ep+\frac{(n^\ep-1)(n^\ep-2\lambda^\ep-2)}{2\lambda^\ep+2}=\frac{n^\ep(n^\ep-1)}{2\lambda^\ep+2}+1=2^{2m},$$ which 
implies that the dimension of $C^\ep$ is $n^\ep-(2m+1)$.
\end{proof}

\begin{proof}[Proof of Theorem C]
This follows immediately from Theorem \ref{t:affcode} and Propositions \ref{p:covC} and \ref{p:dimC}.
\end{proof}

\section{Conway groupoids with large support}\label{s:puzzlarge}

In this section we prove Theorem~D. Although Theorem~D is stated in terms of $\L_\infty(\De)$, it will be convenient to work instead with the hole stabilizer $G=\pi_\infty(\De)$. This approach is advantageous because of the extra flexibility afforded to us from knowing that $G$ is a group. 

In light of this we record the following statement which is equivalent to Theorem~D.

\begin{Thm}\label{t: d}
Suppose that $\De$ is a supersimple $2-(n,4,\lambda)$ design and that $G:=\pi_\infty(\De)$ is the associated hole-stabilizer. Suppose, furthermore, that $[\infty, a, b, \infty]=1$ whenever $\infty$ is collinear with $\{a,b\}$. Then one of the following is true:
\begin{enumerate}
 \item $\De$ is a Boolean design and $G$ is trivial;
 \item $\De=\mathbb{P}_3$ (the projective plane of order $3$) and $G\cong M_{12}$; or 
 \item $G = \Alt(n-1)$.
\end{enumerate}
\end{Thm}

The fact that Theorem~\ref{t: d} is equivalent to Theorem~D can be proved using \cite[Theorem B]{Puzz}, \cite[Proposition 3.4]{Co}, and \cite[Lemma 3.1]{Puzz}. Throughout this section we operate under the suppositions of Theorem~\ref{t: d}.

\subsection{Background results}

We start by collecting a number of important background results.

For a permutation group $H$ acting on a set of size $d$ we write $\mu(H)$ for the smallest number of elements moved by a non-trivial element of $H$ (i.e. $\mu(H)$ is the size of the smallest possible support of a non-trivial element of $H$). In what follows we will use the crucial fact that if $H$ is primitive and doesn't contain $\Alt(d)$, then $\mu(H)$ is bounded below by a function of $d$. 

The following theorem is due to Liebeck and Saxl \cite{ls}, and makes use of the Classification of Finite Simple Groups.

\begin{Thm}\label{t: ls}
Let $d$ be a positive integer and let $H$ be a primitive subgroup of $\Sym(d)$ that does not contain $\Alt(d)$. Either $\mu(H)\geq \frac13 d$ or $(\Alt(m))^r\unlhd G \leq \Sym(m)\wr \Sym(r)$ where $m\geq 5$ and the wreath product acts, via the product action on $\Omega=\Delta^r$ and $\Delta$ is either the set of $\ell$-subsets of $\{1,\dots, m\}$ ($1\leq\ell<\frac12m$) or $m=|\Delta|=6$. In particular, in all cases, $\mu(H)\geq 2 (\sqrt{d}-1).$
\end{Thm}

Observe that Theorem~\ref{t: ls} implies that either $\mu(H)\geq \frac13d$ or else we have that $d=\binom{m}{\ell}^r$ or $6^r$. 

We will also need an elementary result from number theory, which can be regarded as a special case of Mih{\u{a}}ilescu's theorem, formerly the Catalan conjecture \cite{mihailescu}.

\begin{Lem}\label{t: catalan}
Suppose that $a,b,p$ are positive integers, that $a,b>1$ and that $p^a \pm 1=2^b$. Then either $a=1$ or $p=3, a=2.$
\end{Lem}

\begin{proof}
If $a$ is odd then the second factor in $p^a \pm 1=(p \pm 1)(p^{a-1} \mp \ldots + 1)$ is odd. Hence $a-1=0$ in this case. If $a=2t$ for some $t > 0$ then there are two possibilities: firstly, we could have $2^b=(p^{2t}-1)=(p^t-1)(p^t+1)$ and we obtain immediately that $(p,t,a)=(3,1,2)$. Secondly, we could have $2^b=p^{2t}+1$; but since $p^{2t}+1 \equiv 2 \mod 4$, this yields no solutions. 
\end{proof}

%The final result is Ljunggren's solution to one case of the Ljunggren-Nagell problem \cite{ljunggren}.

%\begin{Thm}\label{t: ljunggren}
% Suppose that $x,y,a$ are positive integers, that $a>1$ and that $x^2-x+1=y^a$. Then $(x,y,a)=(19,7,3)$.
%\end{Thm}

\subsection{A structure result}

Our main tool for proving Theorem D will be the following proposition that provides a detailed description of the structure of a design satisfying the suppositions of Theorem D.

\begin{Prop}\label{p:bool}
 Suppose that $\De$ is a supersimple $2-(n,4,\lambda)$ design, and that $G$ contains no non-trivial elements of the form $g=[\infty, a, b, \infty]$ where $\infty\in\overline{a,b}$. Then $\lambda=2^\alpha-1$ for some positive integer $\alpha$, and any two points $a$ and $b$ lie in a unique Boolean $3-(2^{\alpha+1},4,1)$ subdesign $\De_{a,b}$.
Moreover, writing $\Lambda:=\{\overline{a,b} \mid a,b \in \Omega, a \neq b \}$, the pair $(\Omega,\Lambda)$ is a $2-(n,2^{\alpha+1},1)$ design.
\end{Prop}

For a definition of the Boolean $3-(2^k,4,1)$ design we refer the reader to \cite[Section 2]{Puzz}. Notice that when $\alpha=1$, Proposition \ref{p:bool} is true but gives no information: in this case we have $\lambda=1$, the Boolean subdesign $\De_{a,b}$ is the trivial design containing $1$ line and the pair $(\Omega,\Lambda)$ is just the original design $\De$.

\begin{Lem}\label{l:oneall}
Let $(\Omega,\B)$ be a supersimple $2-(n,4, \lambda)$ design, and let $a,b,c$ be distinct points in $\Omega$ such that $c \in \overline{a,b}$ and $[c,a,b,c]=1$.
Then $\overline{a,c}=\overline{a,b}=\overline{b,c}$.
\end{Lem}

\begin{proof} Let $g=[c,a,b,c]$ and $x\in\overline{a,c}$, so $\{a,c,x,y\}$ is a line for some $y\in\Omega\backslash\{a,c,x\}$.  
If $x\notin \overline{a,b}\cup \overline{b,c}$, then $y^g=x$, which is a contradiction. If $x\notin \overline{a,b}\cap \overline{b,c}$,
then one of $\{a,b,x,y\}$ or $\{b,c,x,y\}$ is a line, contradicting pliability. Thus, as $|\overline{a,c}|=|\overline{a,b}|=|\overline{b,c}|=2\lambda+2$, the result holds.
\end{proof}

Let $\mathcal{D}=(\Omega,\B)$ be a $2-(n,4, \lambda)$ design.  Then, for $r,s\in\Omega$, with  $r\neq s$, let
$\mathcal{B}(r,s)$ denote the set of $\lambda$ lines in $\mathcal{B}$ that contain both $r$ and $s$.

\begin{Lem}\label{l:bool}
Let $(\Omega,\B)$ be a supersimple $2-(n,4, \lambda)$ design with the property
that for all distinct pairs $a,b\in\Omega$ and for all $c\in\overline{a,b}$, $[c,a,b,c]=1$.
Then $\mathcal{D}_{a,b}=(\Omega_{a,b},\B_{a,b})$ is an $SQS(2\lambda+2)$,
where $\Omega_{a,b}=\overline{a,b}$ and 
$$\B_{a,b}=\{\B(r,s)\,|\, r,s\in\overline{a,b}, r\neq s\}.$$
Moreover, $\mathcal{D}_{a,b}$ is a Boolean quadruple system of order 
$2^{\alpha+1}$ for some $\alpha > 0$. Consequently, $\lambda=2^\alpha-1$.
\end{Lem}

\begin{proof} Let $y,r,s$ be three distinct points in $\overline{a,b}$. We show that $y,r,s$ lie in a unique element of $\B_{a,b}$. Suppose first that both $a$ and $b$ lie in the set $\{y,r,s\}$, with $r=a$ and $s=b$ say.
As $y\in\overline{a,b}$, there exists a line $\ell\in\mathcal{B}$ (which is necessarily in $\B(a,b)$) that contains all three points, and by pliability,
this line is unique. Secondly, suppose that at most one of $a,b$ lies in $\{y,r,s\}$, so we may assume that $a,b\notin\{r,s\}$.
Then $[r,a,b,r]=[s,a,b,s]=1$, and by Lemma \ref{l:oneall}, 
$\overline{a,r}=\overline{a,b}=\overline{a,s}$, so $s\in\overline{a,r}$.  Now, by supposition, $[s,a,r,s]=1$,
from which we deduce that $\overline{r,s}=\overline{a,b}$. 
Thus $y\in\overline{a,b}\backslash\{r,s\}=\overline{r,s}\backslash\{r,s\}$, and so $y,r,s$ are contained in a line in $\B$ (which is in $\B(r,s)$)
and by pliability, this line is unique. Therefore $\mathcal{D}_{a,b}$ forms an $SQS(2\lambda+2)$, and hence,
a supersimple $2-(2\lambda+2,4,\lambda)$ design.

As $y\in\overline{r,s}$, $[y,r,s,y]=1$ by supposition, and because $y,r,s$ were arbitrary,
we conclude that $\pi_x(\De_{a,b})=1$ for each $x\in\overline{a,b}$.  Hence, 
$\De_{a,b}$ is a Boolean quadruple system of order $2^\alpha$ for some $\alpha > 0$ by \cite[Theorem B]{Puzz}.
\end{proof}

\begin{proof}[Proof of Proposition~\ref{p:bool}]
The first statement of the proposition is a consequence of Lemma \ref{l:bool}. Thus it remains to show that the pair $(\Omega, \Lambda)$ is a 
$2-(n,2^{\alpha+1},1)$ design. But each pair of elements $a,b \in \Omega$ is contained in $\overline{a,b}$ and if there exist another pair $x,y \in \Omega$ such that $a,b \in \overline{x,y}$ then $\overline{x,y}=\overline{a,b}$, as is shown in the proof of Lemma \ref{l:bool}. Consequently $\overline{a,b}$ is the unique element of $\Lambda$ that contains $\{a,b\}$. 
\end{proof}

We record a corollary to Proposition \ref{p:bool}:

\begin{Cor}
Suppose that $\De$ is a supersimple $2-(n,4,\lambda)$ design, and that $G$ contains no non-trivial elements of the form $g=[\infty, a, b, \infty]$ where $\infty\in\overline{a,b}$. If $G$ contains $\Alt(n-1)$ then $G=\Alt(n-1)$ 
\end{Cor}

\begin{proof}
Proposition \ref{p:bool} implies that $\lambda = 2^\alpha-1$ for some positive integer $\alpha$ and, in particular, $\lambda$ is odd. Therefore $G$ is generated by even permutations and since $\Alt(n-1) \leq G \leq \Sym(n-1)$, the result follows.
\end{proof}

\subsection{Proving Theorem D}

Our job now is to prove Theorem D, and to do this we will make heavy use of Proposition~\ref{p:bool}. We will also need to make use of Theorem~E part (2), a short proof of which is given in Section~\ref{s: imprim}. Note that although the proof of part (4) of Theorem~E makes use of Theorem D, the earlier parts do not.

We begin by recording an immediate corollary.

\begin{Cor}\label{c: bool}
Suppose that a hole stabilizer $G=\pi_\infty(\De)$ contains no non-trivial elements of the form $g=[\infty, a, b, \infty]$ where $\infty\in\overline{a,b}$. Suppose, furthermore, that $G$ does not equal $\Alt(n-1)$. Then $\lambda=2^{\alpha}-1$ for some integer $\alpha$ and (setting $k=2^{\alpha+1})$, $n=k, k^2-k+1, 2(k^2-k)+1, k^2$ or $2k^2-k$. If $n=k$ then $G$ is trivial; otherwise $G$ is primitive.
\end{Cor}
\begin{proof}
We apply Proposition~\ref{p:bool} to deduce the existence of a $2-(n,k,1)$ design $(\Omega, \Lambda)$. Suppose that the design is trivial, i.e. $n=k$. Then Proposition~\ref{p:bool} implies that $\De$ is the Boolean design and \cite[Theorem B]{Puzz} implies that $G$ is trivial.
 
Suppose next that the associated $2-(n,k,1)$ design is non-trivial, i.e. $n>k$. Observe that $k=2\lambda+2$ and now Fisher's inequality implies that 
\[n>k^2-k>9\lambda+1. 
\]
 Thus, by Theorem~E (2), $G$ is primitive.

We know that $G$ is generated by elements of the form $[\infty, a, b, \infty]$ and these have support at most $6\lambda+2$. Combining this fact with the inequality $\mu(H)\geq 2 (\sqrt{d}-1)$ of Theorem~\ref{t: ls} (and setting $d=n-1$) we obtain
\[n\leq 9\lambda^2+12\lambda+5<3k(k-1).\]
We also have the conditions that $k-1$ divides $n-1$ and $k(k-1)$ divides $n(n-1)$. Note that $k$ is a power of $2$.

If $n$ is odd, then $k(k-1)$ divides $n-1$ and we conclude that either $n=k^2-k+1$ or $2(k^2-k)+1$. If $n$ is even, then $k-1$ divides $n-1$ and $k$ divides $n$. Hence $n-1=(1+ak)(k-1)$ for some $a > 0$ and we obtain that $n=k^2$ or $2k^2-k$ as required.
\end{proof}

\begin{Lem}\label{l: a}
 Suppose that a hole stabilizer $G=\pi_\infty(\De)$ contains no non-trivial elements of the form $g=[\infty, a, b, \infty]$ where $\infty\in\overline{a,b}$. Suppose, furthermore, that $G$ is neither trivial nor does it equal $\Alt(n-1)$, and that $\lambda>1$.% and that $n>9\lambda^2-12\lambda+5$. 
 Then the following hold:
\begin{enumerate}
 \item $G$ is primitive.
 \item $\lambda=2^\alpha-1$ for some integer $\alpha\geq 2$ and (setting $k=2^{\alpha+1})$, 
 \[n= k^2-k+1, 2(k^2-k)+1, k^2 \textrm{ or } 2k^2-k.\]
 \item There exist integers $m,\ell,r$ ($m\geq 5$, $1\leq\ell<\frac12m$) such that $n-1={m\choose \ell}^r$ or $6^r$. Furthermore $(\Alt(m))^r\unlhd G \leq \Sym(m)\wr \Sym(r)$ where $m\geq 5$ and the wreath product acts, via the product action on $\Omega=\Delta^r$ and $\Delta$ is either the set of $\ell$-subsets of $\{1,\dots, m\}$ or $m=|\Delta|=6$.
\end{enumerate}
\end{Lem}
\begin{proof}
We apply Corollary~\ref{c: bool} and observe that, since $G$ is not trivial, $n\neq k$. Thus $G$ is primitive and (1)  and (2) hold. 

%Suppose next that $\alpha=2$ and $\lambda=3$. Then $k=8$ and we must exclude the values $n=57, 64, 113$ and $120$. In this situation we have non-trivial elements $[\infty, a, b,\infty]$ of support at most $6\lambda+2=20$. On the other hand one can use GAP to check that if $g$ is a non-trivial element in a  primitive group of degree $56$, $63$, $112$ or $119$, then the support of $g$ exceeds $20$. This is a contradiction and we assume from here on that $\alpha>3$ (and so $\lambda\geq 7$). %Observe now that, since $\lambda\geq 7$, both $n^2<9\lambda^2-12\lambda+5$, thus Corollary~\ref{c: bool} implies that (2) holds. 

Now observe that $k=2\lambda+2$ and that $G$ contains non-trivial elements with support of size at most $6\lambda+2=3k-4$. If $\lambda\neq 3$, then all four possible values for $n$ are strictly greater than $9k-11$, hence Theorem~\ref{t: ls} yields (3). 

If $\lambda=3$, then three of the possible values for $n$ are strictly greater than $9k-11=61$ and Theorem~\ref{t: ls} yields (3).  To rule out the final case (when $n=k^2-k+1=57$) we use GAP \cite{GAP} to confirm that none of the primitive groups of degree $56$ contain non-trivial elements with support of size at most $6\lambda+2=20$, thus this situation can be excluded entirely.
\end{proof}

\begin{Lem}\label{l: b}
Let $k=2^{\alpha+1}$ for some integer $\alpha\geq 2$, and suppose that $d=k^2-k$ or $2(k^2-k)$. Then $d\neq 6^r$ and if $d={m\choose \ell}^r$ for positive integers $m,\ell$ and $r$ with $\ell\leq\frac{m}{2}$, then either $(m,\ell,r)=(d,1,1)$ or else $(d,k)=(57,8)$.
\end{Lem}
\begin{proof}
Suppose that $d=s^r$ for some integer $s$ and observe that $d$ is a product of $k-1$ (an odd number) and a power of $2$. Thus $k-1=s_1^r$ for some integer $s_1$. Now Lemma~\ref{t: catalan} implies that $r=1$. One concludes immediately that $d\neq 6^r$.
 
Suppose that $d={m\choose \ell}$. Observe that $d$ is divisible by $2^{\alpha+1}$. It is trivial to observe that if $2^{\alpha+1}$ divides ${m\choose \ell}$, then $m\geq 2^{\alpha+1}$ and hence $$k(k-1) > \frac{k(k-1)\cdots (k-\ell+1)}{\ell!}.
$$
The inequality implies that either $\ell\leq 2$ or $k\leq 8$.

Suppose that $\ell \leq 2$. If $\ell =2$ then $m(m-1)=2^x(2^y-1)$, for some integers $x,y$ with $x > y$ which is absurd. Hence $\ell=1$ and the result follows. 

Finally, suppose that $k\leq 8$ and $\ell>2$. Then one obtains immediately that $k=m=8$, $\ell=3$, $d=57$ and the result follows.
\end{proof}

\begin{Lem}\label{l: c}
Let $k=2^{\alpha+1}$ for some integer $\alpha\geq 2$, and suppose that $d=k^2-1$ or $2k^2-k-1$. Then $d\neq 6^r$ and if $d={m\choose \ell}^r$ for positive integers $m,\ell$ and $r$, then $r=1$.
\end{Lem}
 \begin{proof}
Observe that $d$ is odd, and thus $d\neq 6^r$. Suppose first that
\[d=k^2-1=(k-1)(k+1)=s^r\]
for some positive integers $r$ and $s$. Then, since $k-1$ and $k+1$ are coprime, we conclude that $k-1=s_1^r$ for some positive integer $s_1$. Now Lemma~\ref{t: catalan} implies that $r=1$ as required.%then Theorem~\ref{t: ljunggren} implies that the result holds or else ${m\choose \ell}=7$. But in the latter case we have $k=19$ which is a contradiction.

Assume, then that $d=2k^2-k-1=s^r$ for some integer $r$. There are two cases. First, suppose that $2k+1$ and $k-1$ are coprime. Then $k-1=s_1^r$ for some integer $s_1$ and Lemma~\ref{t: catalan} implies that $r=1$. 

Second, suppose that $2k+1$ and $k-1$ are not coprime; then their highest common factor is $3$ and we conclude, moreover that $\alpha+1$ is even. In this case $k-1=(\sqrt{k}-1)(\sqrt{k}+1)$ and one of these two factors is indivisible by $3$.

Suppose first that $\sqrt{k}-1$ is indivisible by $3$. Then $\sqrt{k}-1$ is coprime to $2k+1$ and $\sqrt{k}+1$ and we conclude that $\sqrt{k}-1=x^r$ for some integer $x$. Now Lemma~\ref{t: catalan} implies that $r=1$ as required.

Suppose finally that $\sqrt{k}+1$ is indivisible by $3$. Then $\sqrt{k}+1$ is coprime to $2k+1$ and $\sqrt{k}-1$ and we conclude that $\sqrt{k}+1=x^r$ for some integer $x$. Now Lemma~\ref{t: catalan} and the fact that $\sqrt{k}+1$ is indivisible by $3$ implies that $r=1$ as required.
\end{proof}

\begin{Lem}\label{l: d}
 Suppose that $G$ is isomorphic to a subgroup of $\Sym(m)$ and consider the natural action of $G$ on the set of $\ell$-subsets of $\{1,\dots,m\}$. Then a non-trivial element of $G$ has support at least $2{{m-2}\choose{\ell-1}}$. 
\end{Lem}
\begin{proof}
 Let $g$ be a non-trivial element of $G$ and let $i$ be an element that is moved by $G$. Thus $i^g=j$ with $j\neq i$. Let $k=j^g$ and observe that, although it is possible to have $i=k$, we know that $j\neq k$.
 
Now observe that any set containing $i$ but not $j$ lies in the support of $g$, and there are ${{m-2}\choose{\ell-1}}$ of these. Similarly any set containing $j$ but not $k$ lies in the support of $g$, and there are ${{m-2}\choose{\ell-1}}$ of these. The two types of set are distinct hence the result follows.
\end{proof}

We remark that if $g\in G$ is a transposition, then the support of $g$ in the given action is of size exactly $2{{m-2}\choose{\ell-1}}$. We are ready to prove Theorem~D.%The previous lemmas together yield the following proposition.

\begin{proof}[Proof of Theorem~D]
If $\lambda=1$, then the result is a consequence of \cite[Theorem C]{Puzz}. If $G$ is trivial, then the result is a consequence of \cite[Theorem B]{Puzz}. Thus we assume that $\lambda>1$ and that $G$ is not trivial and we must show that $G=\Alt(n-1)$. 

Suppose, for a contradiction, that $G$ does not equal $\Alt(n-1)$. Then Lemma~\ref{l: a} implies that $G$ is primitive and, for each value of $\lambda$, gives four possible values for $n$. For two of these values Lemma~\ref{l: b} implies immediately that either $G$ is $\Alt(n-1)$ (and we are done), or else $(n,k)=(57,8)$. Now GAP \cite{GAP} confirms that none of the primitive groups of degree $56$ contain non-trivial elements with support of size at most $6\lambda+2=20$, thus this situation is excluded.

We are left with the possibility that $n=k^2$ or $2k^2-k$ where $k=2\lambda+2\geq 8$. Now Lemma~\ref{l: c} implies that $\Alt(m)\leq G \leq \Sym(m)$ for some $m\geq 5$ and that the action of $G$ on $n-1$ points is isomorphic to the natural action of $G$ on the set of $\ell$-subsets of $\{1,\dots, m\}$. We know that $G$ contains elements with support of size at most $s=6\lambda+2=3k-4$ and we observe that
\[
 n-1\geq k^2-1 \geq \frac19 s^2.
\]
Now Lemma~\ref{l: d} implies that $m$ and $\ell$ satisfy
\[
 {m\choose\ell} \geq \frac49{{m-2}\choose{\ell-1}}^2.
\]
This implies in turn that
\[
 m\geq \frac49 {{m-2}\choose{\ell-1}}
\]
and one concludes immediately that either $m\leq 8$ or $\ell-1=1$. 

Suppose first that $m\leq 8$. Then $n-1={m\choose\ell}\leq 70$ and we conclude that $k=8$ and $n=k^2$. But there does not exist $\ell$ such that $n-1=63= {m\choose\ell}$ for any $m\leq 8$ so this case can be excluded.

Thus we conclude that $\ell=2$. This implies that \[n-1= (2k+1)(k-1)=\frac12 m(m-1)\] and so
\[
 (2k+1)(2k-2)=m(m-1).
\]
Since $(m,m-1)=1$, this is clearly impossible for $k\geq 8$ and the result is proved.
\end{proof}

\section{Properties of Conway groupoids}\label{s:imprim}

In this section we prove Theorem E and throughout we operate under the suppositions of Theorem E. Note that parts of this theorem are already known: when $\lambda=1$ or $2$, Theorem~E is an immediate consequence of \cite[Theorem C]{Puzz}. Furthermore, part (1) of Theorem~E is Lemma 6.1 of \cite{Puzz}. Thus, to prove Theorem~E we can (and will) assume throughout that  $n>4\lambda+1$ and so $G:=\pi_\infty(\De)$ is transitive. 

\subsection{The imprimitive case}\label{s: imprim}

In this section we suppose that $G$ is imprimitive and that $\Delta$ is a block of size $k$; we will prove part (2) of Theorem~E. We need the following result from \cite{Puzz}.

\begin{Lem}\label{l: prim}
Let $n > 4\lambda+1$ and suppose that $G$ preserves a system of imprimitivity with $\ell$ blocks each of size $k$ (so that $n-1=k\ell$). Then at least one of the following holds:
\begin{enumerate}
\item[(i)] if $a,c\in \Omega$ lie in the same block of imprimitivity, then $\infty\in\overline{a,c}$;
\item[(ii)] $n\leq \frac{6\ell}{\ell-1}\lambda+1$.
 \end{enumerate}
\end{Lem}

\begin{proof}[Proof of Theorem E (2)]
Suppose that $n>9\lambda+1$. We assume (for a contradiction) that $G$ preserves a system of imprimitivity with $\ell$ blocks each of size $k$. Suppose first that case (i) of Lemma \ref{l: prim} holds and let $\Delta:=\{c_1,\ldots,c_k\}$ be a block of imprimitivity. Thus there exist points $d_2,\ldots,d_k \in \Omega$ so that $\{\infty,c_1,c_i,d_i\}$ is a line for each $2 \leq i \leq k$. Define: $$\Gamma:=\overline{\infty,c_1} \cup \overline{c_1,d_2} \cup \overline{d_2,\infty},$$ and observe that since $\Delta \subseteq \overline{\infty,c_1}$, $\Delta \subset \Gamma$. Also note that $$|\Gamma| \leq 3(2\lambda+2)-12+4=6\lambda-2 < n.$$ Hence we may choose $e \in \Omega \backslash \Gamma$ and define $g:=[\infty,c_1,e,\infty]$. Now, $\infty \notin \overline{c_1,e}$ so that $c_1^g=e$ and since $e \notin \Delta$, we must have $\Delta^g \cap \Delta = \emptyset.$ Furthermore, since $d_2 \notin \overline{c_1,e} \cup \overline{e,\infty}$, necessarily, $\Delta^g=\{e,d_2,\ldots,d_k\}.$ In particular (by Lemma \ref{l: prim}(ii)) $\infty \in \overline{e,d_2}$. But $e \notin \overline{d_2,\infty}$, a contradiction.

We conclude therefore that case (ii) of Lemma \ref{l: prim} holds, which is possible only if $\ell=2$. This implies that $G$ contains an element of support  of size $2k=n-1$ in its generating set, contradicting the fact that $G$ is generated by elements with support of size at most $6\lambda+2$ (\cite[Lemma 7.3]{Puzz}). This completes the proof.
\end{proof}

\subsection{The primitive case}

In this section we suppose $n$ is large enough so that, by Theorem~E (2), $G$ is primitive and we prove the remaining parts of Theorem~E. We recall that, for a primitive permutation group $H$ we write $\mu(H)$ for the minimal size of the support of a non-trivial element of $G$. Our strategy will be to exploit the fact that hole stabilizers naturally contain elements of small support.

We will make use of the following result of Babai \cite{babai}, which is a weaker version of Theorem~\ref{t: ls} that has the advantage of not depending on the Classification of Finite Simple Groups.

\begin{Thm}\label{t: babai}
Let $d$ be a positive integer and let $H$ be a primitive subgroup of $\Sym(d)$ that does not contain $\Alt(d)$. 
Then we have that $\mu(H)\geq \frac{1}{2} (\sqrt{d}-1).$
\end{Thm}

The following result is part of Lemma 3.1 in \cite{Puzz}.

\begin{Lem}\label{l: kkk}
 $G=\langle [\infty, a, b, \infty] \mid a,b\in\Omega\backslash\infty\rangle$. Furthermore the elements $[\infty, a, b, \infty]$ have support of size at most $6\lambda+2$.
\end{Lem}

\begin{proof}[Proof of Theorem~E]
We have already proved parts (1) and (2): thus we must prove parts (3) and (4).

Suppose that $n>144\lambda^2+120\lambda+26$. Then Theorem~E (2) implies that $G$ is primitive. Suppose that $G$ does not contain $\Alt(n-1)$. Then Theorem \ref{t: babai} and Lemma~\ref{l: kkk} imply that
\[
 6\lambda+2 \geq \frac{1}{2}(\sqrt{n-1}-1).
\]
Rearranging the inequality, one obtains a contradiction as required.

We are left with part (4). If $\lambda\leq 2$, then the result is a consequence of \cite[Theorem C]{Puzz}. Suppose, then, that $\lambda\geq 3$ and that $n>9\lambda^2-12\lambda+5$. Then, in particular, $n>9\lambda+1$ and $G$ is primitive. Suppose that $G$ does not contain $\Alt(n-1)$.

Suppose, first, that $G$ contains a non-trivial element of the form $g=[\infty, a, b, \infty]$ where $\infty\in\overline{a,b}$. Then $g$ has support of size at most $6\lambda-6$ and, combining this fact with the inequality $\mu(H)\geq 2 (\sqrt{d}-1)$ given by Theorem~\ref{t: ls}, we obtain a contradiction and the result is proved. Suppose, on the other hand, that $G$ does contain a non-trivial element of the form $g=[\infty, a, b, \infty]$ where $\infty\in\overline{a,b}$. Then Theorem D gives the result.
\end{proof}

\subsection{The case \texorpdfstring{$\lambda=3$}{Lambda=3}}\label{s: l3}

In previous work with A. Nixon \cite{Puzz} Conway groupoids associated with $2-(n,k,\lambda)$ designs were completely classified for $\lambda\leq 2$. In this subsection we discuss the possibility of extending this classification to deal with the case $\lambda=3$.

We assume throughout this section that $G$ is the hole stabilizer $\pi_\infty(\De)$ of a $2-(n,4,3)$ design. We state two lemmas dealing with the different possibilities for $G$.

\begin{Lem}\label{l: lambda 3 a1}
Suppose that $G$ is primitive. Then either $G \cong \Alt(n-1)$ or one of the following holds:
\begin{itemize}
 \item $n=12$ and $G \in \{M_{11}, \rm{PSL}_2(11), C_{11} \rtimes C_5, C_{11}\}$;
 \item $n=13$ and $G \in \{M_{12},M_{11},\rm{PSL}_2(11)\}$;
 \item $n=16$ and $G \in \{\rm{SL}_4(2),\Sym(6), \Alt(7),\Alt(6)\}$;
 \item $n=17$ and $G$ is isomorphic to one of 19 primitive subgroups of $2^4.\rm{SL}_4(2)$;
% \[
 %\begin{array}{cccccc} 
%\begin{aligned}
% &2^4.PSL_4(2), \, 2^4.S_6, \, (C_2^4) \rtimes S_5, \, 2^4.A_7, \, S_4 \wr S_2, \, 2^4.(S_3 \times S_3), \\
% &A_4 \wr C_2, \, ASL_2(4) \rtimes C_2, \, A%\Gamma L_2(4), \, (C_2^4) \rtimes A_5, \, 2^4.A_6, \, 2^4. ((C_3 \times C_3) \rtimes C_4), \\
% &A\Gamma L_1(16), \, AGL_2(4), \, ASL_2(4), \, (C_2^4) \rtimes D_{10}, \, ((C_2^4) \rtimes C_5).4, \, AGL_1(16) \rtimes C_2;
% \end{array}
%\end{aligned}
%\]
 \item $n=28$ and $G=\PSp_4(3) \rtimes C_2$;
 \item $n=29$ and $G \in \{\Sp_6(2), \Sym(8)\}$.
 \end{itemize}
\end{Lem}
\begin{proof}
Suppose, first, that $G$ does not contain a non-trivial element of the form $g=[\infty, a, b, \infty]$ where $\infty\in\overline{a,b}$. Then Theorem~\ref{t: d} implies that $G\cong\Alt(n-1)$ as required.

Suppose, on the other hand, that $G$ contains a non-trivial element of the form $g=[\infty, a, b, \infty]$ where $\infty\in\overline{a,b}$. Then $G$ contains an element with support of size at most $12$; all primitive groups containing an element with support of size at most $15$ have been known explicitly since long before CFSG (see, especially, \cite{manning, manning1}; we refer to the library in GAP\cite{GAP} for verification).
  
Now, of the list provided by GAP we are able to exclude all of these groups that are not subgroups of $\Alt(n-1)$ and, for $n>9$, the resulting groups are those listed in the lemma. The remaining values -- when $n=8$ or $9$ -- can be excluded directly since there is only one supersimple design in each case, and neither yield a primitive hole stabilizer.
\end{proof}

Note that Lemma~\ref{l: lambda 3 a1} lists possible isomorphism types for $\pi_\infty(\De)$. We do not know whether designs exist yielding hole-stabilizers of these forms.

\begin{Lem}\label{l: lambda 3 a2}
Suppose that $G$ is intransitive. Then $n=8$ and $G$ is trivial, or else $n=12$ or $13$.
Suppose that $G$ is transitive and imprimitive. Then $n=9$ and $G\cong \Alt(4)\wr C_2$, or else $n=13, 17, 21, 25$ or $28$.
\end{Lem}
\begin{proof}
Note first that, using the Handbook of Combinatorial Designs \cite{Handbook}), it is easy to confirm that the two case, $n=8$ and $n=9$, each yield exactly one supersimple $2-(n,4,3)$ design. When $n=8$ this design is the Boolean one and the associated hole stabilizer is trivial; when $n=9$ the associated hole stabilizer is $\Alt(4)\wr C_2$, a transitive, imprimitive group, as required. Assume now that $n>9$.

If $G$ is intransitive, then the result follows from Theorem~E (1).  Now suppose that $G$ is transitive and imprimitive. Then Theorem~E (2) implies that $n\leq28$. To complete the proof we use the fact that if a $2-(n,4,3)$ design exists, then $n\equiv 0,1\pmod 4$ and, furthermore, that, since $G$ is imprimitive, $n-1$ is not a prime.
\end{proof}

Lemmas~\ref{l: lambda 3 a1} and \ref{l: lambda 3 a2} imply that the job of classifying Conway groupoids associated with $2-(n,4,3)$ designs is reduced to the situation where $12\leq n\leq 29$.

\end{document}